\documentclass[12pt]{article}

\usepackage{amsmath}
\usepackage{amsfonts}
\usepackage{amssymb}
\usepackage{verbatim}
\usepackage{caption}
\usepackage{graphicx}
\usepackage{subfigure}
\usepackage{tabu}
\usepackage{multirow}
\usepackage{array}
\usepackage{wrapfig}
\graphicspath{ {Images/} }
\usepackage{longtable}
\usepackage[english]{babel}
\usepackage{afterpage}
\usepackage {tikz}
\usetikzlibrary {positioning}
\usepackage {xcolor}
\usepackage{mathtools}

\newtheorem{theorem}{Theorem}[section]
\newtheorem{lemma}[theorem]{Lemma}
\newtheorem{definition}[theorem]{Definition}

\newtheorem{corollary}[theorem]{Corollary}

\newenvironment{proof}{\paragraph{\textbf{Proof}}}
 {\nopagebreak\hfill\nopagebreak\rule{2mm}{2mm}\par\bigskip}
\newcommand*{\QED}{\nopagebreak\hfill\nopagebreak\rule{2mm}{2mm}\par\bigskip}%

\newcommand{\m}[1]{\text{ (mod } #1 \text{)}}%

\newcounter{CaseCount}
\begin{document}

\title{The Domination Equivalence Classes of Paths}
\author{J.I.~Brown and I. Beaton\\
Department of Mathematics and Statistics \\
Dalhousie University, Halifax, NS B3H 3J5, Canada\\
\date{}
}
\maketitle

\begin{abstract}
A dominating set $S$ of a graph $G$ of order $n$ is a subset of the vertices of $G$ such that every vertex is either in $S$ or adjacent to a vertex of $S$. 
The domination polynomial is defined by $D(G,x) = \sum d(G,i)x^i$ where $d(G,i)$ is the number of dominating sets in $G$ with cardinality $i$. Two graphs $G$ and $H$ are considered $\mathcal{D}$-equivalent if $D(G,x)=D(H,x)$. 
The equivalence class of $G$, denoted $[G]$, is the set of all graphs $\mathcal{D}$-equivalent to $G$. 
Extending previous results, we determine the equivalence classes of all paths.
\end{abstract}


\section{Introduction}

Let $G=(V,E)$ be a graph. 
A set $S$ of the vertex set $V$ of graph $G$ is a \textbf{dominating set} if 
for each $v \in V(G)$, either $v \in S$ or there exists $u \in S$ which is adjacent to $v$. The \textbf{domination number} of $G$, denoted $\gamma(G)$, is the cardinality of the smallest dominating set of $G$.
There is a long history of interest in domination in both  pure and applied settings \cite{funddom,Hooker}.

As for many graph properties, one can more thoroughly investigate domination via a generating function. Let $\mathcal{D}(G,i)$ be the collection of dominating sets of a graph $G$ with cardinality $i$ and let $d(G,i) = |\mathcal{D}(G,i)|$. Then the \textbf{domination polynomial} $D(G,x)$ of $G$ is defined as
$$D(G,i) = \sum_{i=\gamma (G)}^{|V(G)|} d(G,i)x^i.$$
 
A natural question to ask is to what extent can a graph polynomial describe the underlying graph (for example, a survey of what is known with regards to chromatic polynomials can be found in \cite[ch. 3]{dongbook}. We say that two graphs $G$ and $H$ are \textbf{domination equivalent} or simply ${\mathcal{D}}$\textbf{-equivalent} (written $G \sim_{\mathcal{D}} H$) if they have the same domination polynomial.
As in \cite{2010Char}, we let $[G]$ denote the ${\mathcal{D}}$-equivalence class determined by $G$, that is $[G] = \{H | H \sim_{\mathcal{D}} G\}$. A graph $G$ is said to be \textbf{dominating unique} or simply \textbf{$\mathcal{D}$-unique} if $[G] = \{H| H \text{ is ismorphic to } G\}$.

Two problems arise: Which graphs are $\mathcal{D}${\textbf{-unique}}, that is, are completely determined by their domination polynomials? More generally, can we determine the $\mathcal{D}$-equivalence class of a graph?
Both problems appear difficult, but there are some partial results known.
In \cite{2014EqCycleAll} Akbari and Oboudi showed all cycles are $\mathcal{D}$-unique. Anthony and Picollelli classified all complete $r$-partite graphs which are $\mathcal{D}$-unique in \cite{2014EqBipartAll}. In \cite{2011Order10} Alikhani and Peng showed most cubic graphs of order 10 (including the Peterson graph) are $\mathcal{D}$-unique.
In \cite{2012Recurr} Kotek, Preen, and Simon defined and characterized irrelevant edges. These are edges which can be removed without changing the domination polynomial of a graph. From this they could show various trees (in particular paths \cite{2010Char}) barbell graphs \cite{2015EqGen}, and other graphs are not $\mathcal{D}$-unique.

In \cite{2010Char} Akbari, Alikhani and Peng considered the ${\mathcal{D}}$-equivalence classes of paths, and showed that $[P_n] = \{P_n, \widetilde{P_n}\}$ for $n \equiv 0 \pmod{3}$ where $\widetilde{P_n}$ is the graph obtained by added an edge between the two stems in $P_n$. In this paper we extend this result and determine the ${\mathcal{D}}$-equivalence class for path $P_{n}$ for {\textit all} $n$.

A few definitions are in order before we begin. The \textbf{order} of a graph is its number of vertices; $P_{n}$ and $C_{n}$ denote the \textbf{path} and \textbf{cycle} of order $n$, respectively. The set of vertices $N(v) = \{u|uv \in E(G)\}$ 
is called the \textbf{open neighbourhood} of $v$; similarly, $N[v] = N(v) \cup \{v\}$ is called the \textbf{closed neighbourhood} of $v$ (clearly 
$\displaystyle{N[S] = \cup_{v \in S} N[v]}$. A vertex of degree $1$ is a \textbf{ leaf}, its neighbour is a \textbf{stem}, and the edge between them is called a \textbf{pendant edge}. 

\section{Coefficients of Domination Polynomials}
\label{sec:Coeff}

We start with an examination of what the domination polynomial encodes about graphs in general, and about paths in particular. 
Some coefficients of domination polynomials are know for general graphs. In \cite{2014Intro}, Alikhani and Peng determined $d(G,n-1)$ and $d(G,n-2)$ in terms of certain properties within the graph. In this section we will give general (though involved) formulae for $d(G,n-3)$ and $d(G,n-4)$, as well as derive some properties of the coefficients for the domination polynomials of paths. All of these will be helpful in the following section where we determine the equivalence class for paths.

\begin{theorem}
\label{thm:DomPolyGeo}

\textnormal{\cite{2014Intro}} Let $G$ be a graph of order $n$ with $t$ vertices of degree one and $r$ isolated vertices. If $D(G,x) = \sum_{i=1}^n d(G,i)x^i$ is its domination polynomial then the following hold:
\renewcommand{\labelenumi}{(\roman{enumi})}
 \begin{enumerate}
   \item $d(G,n-1) = n - r$.
   \item $d(G,n-2) = {n \choose 2} - t$ if $G$ has no isolated vertices and no $K_2$-components.
 \end{enumerate}
\QED 
\end{theorem}

When counting the number of dominating sets with cardinality close to $n$, it is sometimes simpler to count the number of subsets which are not dominating. A subset $S \subseteq V(G)$ is not dominating if there exists a vertex $v$ in $G$ such that none of its neighbours, nor itself, is in $S$. That is, $N[v] \cap S = \emptyset$. The next elementary lemma will help us identify which subsets are not dominating.

\begin{lemma}
\label{lem:notdomcomp}
For a graph $G$ and $S \subseteq V(G)$, $S$ is not dominating if and only if there exists a vertex $v \in \overline{S} = V(G)-S$ which is encompassed by $\overline{S}$.
\QED
\end{lemma}

%
%

We now determine the number of dominating sets of a certain size by counting the number of subsets of vertices with a given cardinality which contain the closed neighbourhood of one of its vertices. 
We focus on graphs $G$ with no isolated vertices and no $K_2$ components, as these are what will arise in the next section when considering graphs that are domination equivalent to paths. We first define some graph parameters and subsets.


\begin{enumerate}
\item[$\bullet$] $T_r$: The set of vertices of degree $r$ in $G$ {\em which are not stems}.
\item[$\bullet$] $\omega $:  The number of stems in $G$.\item[$\bullet$] $W = \{s_{1},s_{2},\ldots,s_{\omega}\}$: The set of all stems in $G$.
\item[$\bullet$] $S_i$:  The set of leaves attached to  stem $s_{i}$.
\item[$\bullet$] $H_k$:  A generic subset of vertices in $G$ of cardinality $k$ (\emph{a $k$-subset} of $G$).
\item[$\bullet$] $f_G(H_k,U)$: The number of vertices of $U$ which are encompassed by $H_k$ (such vertices are by necessity in $H_{k}$).
\end{enumerate}

\begin{figure}[h]
\def\c{0.7}
\def\r{2}
\centering
\scalebox{\c}{
\begin{tikzpicture}
\begin{scope}[every node/.style={circle,thick,draw}]
    \node (1) at (0*\r,0*\r) {$l_1$};
    \node (2) at (0*\r,1*\r) {$l_2$};
    \node (3) at (0*\r,-1*\r) {$l_3$};
    \node (4) at (1*\r,0*\r) {$s_1$};
    \node (5) at (2*\r,0*\r) {$s_2$};
    \node (6) at (1.5*\r,-1*\r) {$l_4$};
    \node (7) at (2.5*\r,-1*\r) {$l_5$};
    \node (8) at (2.5*\r,1*\r) {$v_1$};
    \node (9) at (2*\r,1.5*\r) {$v_2$};
    \node (10) at (1.5*\r,1*\r) {$v_3$};
    \node (11) at (3*\r,0*\r) {$v_4$};
    \node (12) at (4*\r,0*\r) {$s_3$};
    
    \node (13) at (3.5*\r,-1*\r) {$l_6$};
    \node (14) at (4.5*\r,1*\r) {$v_5$};
    \node (15) at (5.25*\r,1.5*\r) {$v_6$};
    \node (16) at (6*\r,0.75*\r) {$v_7$};
    \node (17) at (6*\r,-0.75*\r) {$v_8$};
    \node (18) at (5.25*\r,-1.5*\r) {$v_9$};
    \node (19) at (4.5*\r,-1*\r) {$v_{10}$};
    \node (20) at (5*\r,0.5*\r) {$v_{11}$};
    \node (21) at (5*\r,-0.5*\r) {$v_{12}$};
    
    \node (22) at (3.5*\r,1*\r) {$s_4$};
    \node (23) at (3.5*\r,1.5*\r) {$l_7$};
    
\end{scope}

\begin{scope}
    \path [-] (1) edge node {} (4);
    \path [-] (2) edge node {} (4);
    \path [-] (3) edge node {} (4);
    \path [-] (4) edge node {} (5);
    
    \path [-] (5) edge node {} (6);
    \path [-] (5) edge node {} (7);
    \path [-] (5) edge node {} (8);
    \path [-] (8) edge node {} (9);
    \path [-] (9) edge node {} (10);
    \path [-] (10) edge node {} (5);
    
    \path [-] (5) edge node {} (11);
    \path [-] (11) edge node {} (12);
    
    \path [-] (12) edge node {} (13);
    \path [-] (12) edge node {} (14);
    \path [-] (14) edge node {} (15);
    \path [-] (15) edge node {} (16);
    \path [-] (16) edge node {} (17);
    \path [-] (17) edge node {} (18);
    \path [-] (18) edge node {} (19);
    \path [-] (19) edge node {} (12);

    \path [-] (12) edge node {} (20);
    \path [-] (20) edge node {} (21);
    \path [-] (21) edge node {} (12);
    
    \path [-] (12) edge node {} (22);
    \path [-] (22) edge node {} (23);

\end{scope}
\end{tikzpicture}}
\caption{An example of a graph}%
\label{fig:EqG2m}%
\end{figure}
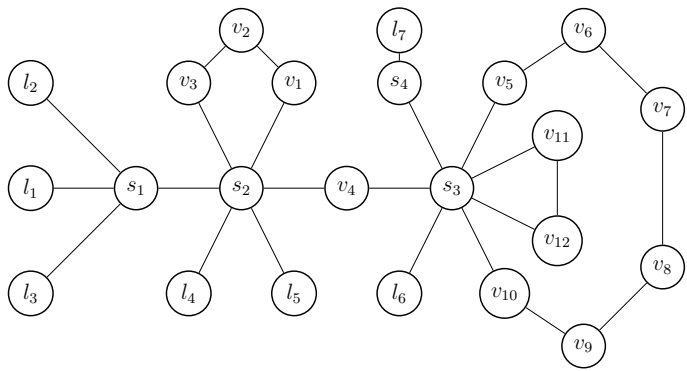

As an example, consider the graph in Figure~\ref{fig:EqG2m}. The set of stems $W$ is $\{s_1, s_2, s_3,s_4\}$ and $\omega = 4$. There are no degree zero vertices therefore $T_0 = \emptyset$. There are six degree one vertices (leaves), none of which are stems, therefore $T_1 = \{l_1,l_2, l_3, l_4, l_5, l_6,l_7\}$. There are 13 degree two vertices, one of which ($s_4$) is a stem, so $T_2 = \{v_i|1 \leq i \leq 12\}$.  The sets of leaves are $S_1 = \{l_1, l_2, l_3\}$, $S_2 = \{l_4, l_5\}$, $S_3 = \{l_6\}$, and $S_4 = \{l_7\}$. An example of an $8$-subset is $H_8 = \{s_3, s_4, l_1, l_6, l_7, v_2, v_{11}, v_{12}\}$. The vertices which are encompassed by $H_8$ are $s_4$, $l_6$, $l_7$, $v_{11}$, and $v_{12}$. Therefore $f_G(H_8,V)=5$ and $f_G(H_8,V-W)=4$ as $s_4 \in W$.


We are now ready to count the number of dominating sets of a given size in a graph without a $K_2$ component.
 
\begin{lemma}
\label{lem:d(G,n-k)}
For a graph $G$ of order $n$ with no $K_2$ components and $k \in \mathbb{N}$, where $2 \leq k \leq n-\gamma(G)$, 

$$d(G,n-k) = {n \choose k} -\bigg(\sum_{i=0}^{k-1}|T_i|{n-i-1 \choose k-i-1}-\sum_{\substack{H_k \subseteq V \\ |H_k|=k}} \mbox{max}(f_G(H_k,V-W)-1,0)\bigg).$$
\end{lemma}

\begin{proof}

As there are ${n \choose k}$ $k$-subsets of vertices in $G$, ${n \choose k}-d(G,n-k)$ is the number of $(n-k)$-subsets of $G$ which are not dominating. Thus by Lemma \ref{lem:notdomcomp}, the number of $(n-k)$-subsets which are not dominating is equivalent to the number of $k$-subsets which encompass at least one vertex. Therefore it is sufficient to show the number of $k$-subsets which encompass at least one vertex is 

\begin{eqnarray}
\sum_{i=1}^{k-1}|T_i|{n-i-1 \choose k-i-1}-\sum_{\substack{H_k \subseteq V \\ |H_k|=k}} \text{max}(f_G(H_k,V-W)-1,0). \label{dnminusk}
\end{eqnarray}

\noindent For each vertex $v \in T_i$, its closed neighbourhood has order $i+1$. If $i \leq k-1$ then there are ${n-i-1 \choose k-i-1}$ $k$-subsets which encompass $v$. If $i > k-1$ then $|N[v]|>k$ and no $k$-subset can encompass $v$. Therefore the first term's count includes every $k$-subset which encompasses a non-stem vertex. We omit the stems of $G$ as any $k$-subset which encompasses a stem $s$ must also encompass one of its leaves $l$ as $N[l] \subseteq N[s]$. Hence each of these $k$-subsets are counted when we count every $k$-subset which contains $l$.

If a $k$-subset $H_k$ encompasses at least one vertex, we wish only to count it once. However, our first term counts each $k$-subset for each non-stem vertex it encompasses. That is, each $H_k$ is counted $f_G(H_k,V-W)$ times and hence over-counted $f_G(H_k,V-W)-1$ times. In the case where $f_G(H_k,V-W) \leq 1$ we have not over counted. As this implies $f_G(H_k,V-W)-1 \leq 0$, it is sufficient to subtract $\mbox{max}(f_G(H_k,V-W)-1,0)$ for each $H_k$ of $G$. This gives us the second term.

\end{proof}

The value of $k$ puts restrictions on both $f_G(H_k,V-W)$ and $H_k$. As $H_k$ has order $k$, it can encompass at most $k$ vertices and hence $f_G(H_k,V-W) \leq k$. Furthermore, if $f_G(H_k,V-W)>1$ then $H_k$ encompasses a vertex $v$, so $N[v] \subseteq H_k$ and therefore $|N[v]| \leq k$, and it follows that any vertex $H_k$ encompasses must have degree less than $k$. In the next lemma we will use Lemma \ref{lem:d(G,n-k)} to determine $d(G,n-3)$ for a graph $G$ of order $n$ with no isolated vertices and no $K_2$ components. Before we begin, we need yet a few more definitions. An \textbf{$\mathbf{r}$-loop} is an induced $r$-cycle in $G$ such that all but one vertex has degree two in $G$. 
Examples of $r$-loops can be found in Figure~\ref{fig:EqG2m}; the vertices $s_3$, $v_{11}$, and $v_{12}$ form a $3$-loop, and the vertices $s_3, v_{5}, v_{6}, \ldots, v_{10}$ form a $7$-loop. The vertices $s_4$, $v_{1}$, $v_{2}$, and $v_{3}$ also form a $4$-loop. Further, we use the following notation, all with respect to a graph $G$:

\begin{enumerate}
\item[$\bullet$] $\mathcal{L}_r$: The set of $r$-loop.
\item[$\bullet$] $\mathcal{L}_r^i$: The set of $r$-loop subgraphs which contain stem $s_i$.
\item[$\bullet$] $\mathcal{C}_r$: The set of components which are cycles of order $r$.
\end{enumerate}

\begin{theorem}
\label{thm:d(G,n-3)}
For a graph $G$ of order $n$ where $G$ has no isolated vertices and no $K_2$ components,

$$d(G,n-3) = {n \choose 3} -\bigg(|T_1| \cdot (n-2)+|T_2|-\sum_{i=1}^{\omega} {|S_i| \choose 2}-|\mathcal{L}_3|-2|\mathcal{C}_3|\bigg).$$

\end{theorem}

\begin{proof}

By Lemma \ref{lem:d(G,n-k)} we know

$$d(G,n-3) = {n \choose 3} -\bigg(\sum_{i=0}^{2}|T_i|{n-i-1 \choose 3-i-1}-\sum_{H_3 \subseteq V} \text{max}(f_G(H_3,V-W)-1,0)\bigg).$$

\noindent As $G$ has no isolated vertices, $|T_0|=0$ and $\displaystyle{\sum_{i=0}^{2}|T_i|{n-i-1 \choose 3-i-1} = |T_1| \cdot (n-2)+|T_2|}$. Now it is sufficient to show

\begin{center}
\begin{equation} \label{eq:pf_d(G,n-3)}
\sum_{H_3 \subseteq V} \text{max}(f_G(H_3,V-W)-1,0) = \sum_{i=1}^{\omega} {|S_i| \choose 2} +|\mathcal{L}_3|+2|\mathcal{C}_3|.
\end{equation}
\end{center}

\noindent Of course, $\mbox{max}(f_G(H_3,V-W)-1,0)$ is only non-zero when $f_G(H_3,V-W) \geq 2$. Therefore we wish to find $3$-subsets of $G$ which encompass two or more non-stem vertices. Let $H = \{u,v,w\}$ be a $3$-subset of $G$ which encompass two or more non-stem vertices. As $H$ has order three, the non-stem vertices which it encompasses have degree at most two. As $G$ has no isolated vertices then the vertices which $H$ encompass are either all in $T_1$ or all in $T_2$ or both. Let $S$ be the set of vertices which $H$ encompasses. We now count each $H$ in the three aforementioned cases.

\vspace{5mm}
\noindent \textbf{Case 1:} \emph{$S \cap T_1 \neq \emptyset$ and $S \cap T_2 = \emptyset$}
\vspace{5mm}

\noindent As $H$ has order three then it either encompasses two or three vertices in $T_1$. Without loss of generality let $u$ and $v$ be two of the vertices encompassed by $H$. Then $u,v \in T_1$, and, as $G$ has no $K_2$ components, $u$ and $v$ are not adjacent. Furthermore as $\text{deg}(u)=\text{deg}(v)=1$ and $H$ encompasses both $u$ and $v$, $N(u)=\{w\}$ and $N(v)=\{w\}$. Therefore $u$ and $v$ are each leaves on the same stem $w$.

\vspace{5mm}
\noindent \textbf{Case 2:} \emph{$S \cap T_1 = \emptyset$ and $S \cap T_2 \neq \emptyset$}
\vspace{5mm}

\noindent As $H$ has order three then it either encompasses two or three vertices in $T_2$. Without loss of generality let $u$ and $v$ be two of the vertices encompassed by $H$. Then $u,v \in T_2$, $N(u)=\{v,w\}$ and $N(v)=\{u,w\}$. Therefore $H$ induces a $3$-cycle in $G$ and $\{u,v\} \subseteq N(w)$. This leaves us with two possibilities for $H$ -- either $\mbox{deg}(w)=2$ or $\mbox{deg}(w)>2$. If $\mbox{deg}(w)=2$ then $H$ is a $3$-cycle component of $G$, and if $\mbox{deg}(w)>2$ then $H$ is a $3$-loop in $G$.

\vspace{5mm}
\noindent \textbf{Case 3:} \emph{$S \cap T_1 \neq \emptyset$ and $S \cap T_2 \neq \emptyset$}
\vspace{5mm}

\noindent We claim this case is impossible. If $S$ contains at least one vertex from both $T_1$ and $T_2$ then let $v \in T_2$ with $N(v)=\{u,w\}$. As $H$ encompasses $v$, $u$ and $w$ are both vertices in $H$. Moreover, as $H$ is order three and $S$ contains at least one vertex in $T_1$, either $u \in T_1$ or $w \in T_1$. Without loss of generality let $u \in T_1$. Then $u$ is a leaf and has only one neighbour, $v$. However, then $v$ is a stem, and by definition not in $T_2$, which is a contradiction.

\vspace{5mm}

Our three cases have produced three possible $3$-subsets which encompass two or more non-stem vertices: two leafs on the same stem, $3$-cycle components, and $3$-loops. The three cases for $H$ are shown in Figure \ref{fig:H3}. Note the vertices which are encompassed by $H$ are shaded.

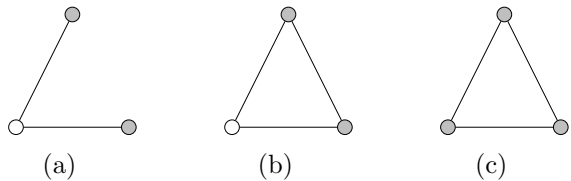
\begin{figure}[!h]
\def\c{0.5}
\centering
\subfigure[]{
\scalebox{\c}{
\begin{tikzpicture}
\node[shape=circle,draw=black,fill=white] (1) at (0,0) {};
\node[shape=circle,draw=black,fill=gray!50] (2) at (3,0) {};
\node[shape=circle,draw=black,fill=gray!50] (3) at (1.5,3) {};

\begin{scope}
    \path [-] (1) edge node {} (2);
    \path [-] (1) edge node {} (3);
\end{scope}
\end{tikzpicture}}}
\qquad
\subfigure[]{
\scalebox{\c}{
\begin{tikzpicture}
\node[shape=circle,draw=black,fill=white] (1) at (0,0) {};
\node[shape=circle,draw=black,fill=gray!50] (2) at (3,0) {};
\node[shape=circle,draw=black,fill=gray!50] (3) at (1.5,3) {};

\begin{scope}
    \path [-] (1) edge node {} (2);
    \path [-] (2) edge node {} (3);
    \path [-] (3) edge node {} (1);
\end{scope}
\end{tikzpicture}}}
\qquad
\subfigure[]{
\scalebox{\c}{
\begin{tikzpicture}
\node[shape=circle,draw=black,fill=gray!50] (1) at (0,0) {};
\node[shape=circle,draw=black,fill=gray!50] (2) at (3,0) {};
\node[shape=circle,draw=black,fill=gray!50] (3) at (1.5,3) {};

\begin{scope}
    \path [-] (1) edge node {} (2);
    \path [-] (2) edge node {} (3);
    \path [-] (3) edge node {} (1);
\end{scope}
\end{tikzpicture}}}
\caption{Every $3$-subset which encompasses two or more non-stems}%
\label{fig:H3}%
\end{figure}

Now we need only to sum $f_G(H,V-W)-1$ for each $3$-subset. We will sum each $f_G(H,V-W)-1$ by evaluating $f_G(H,V-W)-1$ each case then multiplying it by the number of times it occurs in $G$.

If $H$ is two leafs on the same stem then $f_G(H,V-W)-1=1$. This $3$-subset will occur ${|S_i| \choose 2}$ times for each stem. If $H$ is a $3$-loop then $f_G(H,V-W)-1=1$. This $3$-subset will occur $|\mathcal{L}_3|$ times. If $H$ is a $3$-cycle component then $f_G(H,V-W)-1=2$. This $3$-subset will occur $|\mathcal{C}_3|$ times. Taking the sum of each of the cases gives use the right hand side of equation (\ref{eq:pf_d(G,n-3)}).
\end{proof}

We now introduce a new collection of graphs.

\begin{definition}
Let $\mathcal{G}_k$ denotes the set of all graph $G$ with the property that every vertex is either a stem or has degree at most $k$.
\end{definition}

Our focus will be when $k=2$. Two familiar families of graphs in $\mathcal{G}_2$ are paths and cycles. Another example was shown in Figure \ref{fig:EqG2m}. For a graph $G$ of order $n$, clearly $G \in \mathcal{G}_2$ if and only if $\omega+|T_1|+|T_2|=n$. Note if $G \in \mathcal{G}_k$ and $G$ has a $r$-loop then the one vertex of the $r$-loop which is not degree two is a stem.

In the next lemma we will extend our work of Theorem~\ref{thm:d(G,n-3)} and determine $d(G,n-4)$ for a graph $G \in \mathcal{G}_2$ of order $n$ with no isolated vertices and no $K_2$ components (we will make essential use of this in the next section). The proof, although similar to that of Theorem~\ref{thm:d(G,n-3)}, is more involved. Before we begin, we will partition $T_2$ into subsets based on the number of neighbouring stems.


\begin{enumerate}
\item[$\bullet$] $V_0$: The subset of $T_2$ with no adjacent stems.
\item[$\bullet$] $V_1^{i}$:  The subset of $T_2$ adjacent to exactly one stem, stem $i$.
\item[$\bullet$] $V_2^{ij} $:  The subset of $T_2$ adjacent to exactly two stems, stems $i$ and $j$ (denoted $V_2$ when $G$ only has two stems).
\end{enumerate}

\begin{theorem}
\label{thm:d(G,n-4)}
Let $G \in \mathcal{G}_2$ be a graph of order $n$ with no isolated vertices and no $K_2$ components. Then

$$d(G,n-4) = {n \choose 4} -\bigg(|T_1| {n-2\choose 2} + |T_2|(n-3) - \alpha_1 -\alpha_2-\alpha_3     \bigg)$$

\noindent where

\begin{center}
\begin{tabular}{ l c l }
$\alpha_1$ & $=$ & $\sum\limits_{i=1}^{\omega} {|S_i| \choose 2}(n-|S_i|-1)+\sum\limits_{i=1}^{\omega} \frac{|S_i|}{2}(|T_1|-|S_i|)+2\sum\limits_{i=1}^{\omega} {|S_i| \choose 3},$ \\ 

$\alpha_2$ & $=$ & $\sum\limits_{i=1}^{\omega} |V_1^{i}||S_i|+\sum\limits_{i \neq j} |V_2^{ij}|(|S_i|+|S_j|), \mbox{ and}$ \\ 

$\alpha_3$ & $=$ & $|V_0|+\sum\limits_{i=1}^{\omega} \frac{|V_1^i|}{2}+\sum\limits_{i \neq j} {|V_2^{ij}| \choose 2}-|\mathcal{C}_4|+|\mathcal{C}_3|(2n-9)+\sum\limits_{i=1}^{\omega}|\mathcal{L}_3^i|(n-4-|S_i|).$ \\

\end{tabular}
\end{center}

\end{theorem}

\begin{proof}
By Lemma \ref{lem:d(G,n-k)} we know

$$d(G,n-4) = {n \choose 4} -\bigg(\sum_{i=0}^{3}|T_i|{n-i-1 \choose 4-i-1}-\sum_{H_4 \subseteq V} max(f_G(H_4,V-W)-1,0)\bigg).$$

\noindent As $G$ has no isolated vertices, $|T_0|=0$ and $\displaystyle{\sum_{i=0}^{3}|T_i|{n-i-1 \choose 4-i-1} = |T_1| {n-2\choose 2} + |T_2|(n-3)}$ (since $G \in \mathcal{G}_2$ implies $|T_3|=0$). Now it is sufficient to show

\begin{center}
\begin{equation} \label{eq:pf_d(G,n-4)}
\sum_{H_4 \subseteq V} \mbox{max}(f_G(H_4,V-W)-1,0) = \alpha_1+\alpha_2+\alpha_3.
\end{equation}
\end{center}

\noindent Now $\mbox{max}(f_G(H_4,V-W)-1,0)$ from the left hand side of equation (\ref{eq:pf_d(G,n-4)}) is only non-zero when $f_G(H_4,V-W) \geq 2$. Therefore we wish to find $4$-subsets of $G$ which encompass two or more non-stem vertices. Let $H = \{w,x,y,z\}$ be an arbitrary $4$-subset of $G$ which encompass two or more non-stem vertices. As $G \in \mathcal{G}_2$, each non-stem vertex has degree at most two. Furthermore, as $G$ has no isolated vertices, the vertices which $H$ encompass are either all in $T_1$ or all in $T_2$ or both. Let $S$ be the set of non-stem vertices which $H$ encompasses. We now count each $H$ in  three cases, as in the proof of Theorem~\ref{thm:d(G,n-3)}.

\vspace{5mm}
\noindent \textbf{Case 1:} \emph{$S \cap T_1 \neq \emptyset$ and $S \cap T_2 = \emptyset$}
\vspace{5mm}

\noindent As $H$  encompasses at least two non-stem vertices then, without loss of generality, let $w$ and $x$ be encompassed by $H$. Then $w,x \in T_1$ and as $G$ has no $K_2$ components then $w$ and $x$ are not adjacent. Therefore they are either leafs on the same stem, or leaves on different stems. If $w$ and $x$ are leaves on different stems then $y$ and $z$ are each stems and the only non-stems $H$ encompasses are $w$ and $x$. If $w$ and $x$ are leaves on the same stem then, without loss of generality, let $y$ be the stem adjacent to $w$ and $x$. Then $z$ is either a third leaf on $y$ or not. If $z$ is a leaf on the stem $y$ then $H$ encompasses $x$, $w$ and $z$. If $z$ is not a leaf on $y$ then $H$ encompasses only $w$ and $x$. The subgraphs which $H$ induce are shown in Figure \ref{fig:H4T1}. Dark gray vertices are stems, light gray vertices are encompassed by $H$ and dashed edges are possible edges.

\begin{figure}[!h]
\def\c{0.5}
\centering
\subfigure[]{
\scalebox{\c}{
\begin{tikzpicture}
\node[shape=circle,draw=black,fill=gray!25] (1) at (0,0) {$w$};
\node[shape=circle,draw=black,fill=gray!25] (2) at (2,0) {$x$};
\node[shape=circle,draw=black,fill=gray] (3) at (2,2) {$y$};
\node[shape=circle,draw=black,fill=gray] (4) at (0,2) {$z$};

\begin{scope}
    \path [-] (1) edge node {} (4);
    \path [-] (2) edge node {} (3);
    \path [dashed] (3) edge node {} (4);
\end{scope}
\end{tikzpicture}}}
\qquad
\subfigure[]{
\scalebox{\c}{
\begin{tikzpicture}
\node[shape=circle,draw=black,fill=gray!25] (1) at (0,0) {$w$};
\node[shape=circle,draw=black,fill=gray!25] (2) at (2,0) {$x$};
\node[shape=circle,draw=black,fill=gray] (3) at (2,2) {$y$};
\node[shape=circle,draw=black,fill=gray!25] (4) at (0,2) {$z$};

\begin{scope}
    \path [-] (1) edge node {} (3);
    \path [-] (2) edge node {} (3);
    \path [-] (3) edge node {} (4);
\end{scope}
\end{tikzpicture}}}
\qquad
\subfigure[]{
\scalebox{\c}{
\begin{tikzpicture}
\node[shape=circle,draw=black,fill=gray!25] (1) at (0,0) {$w$};
\node[shape=circle,draw=black,fill=gray!25] (2) at (2,0) {$x$};
\node[shape=circle,draw=black,fill=gray] (3) at (2,2) {$y$};
\node[shape=circle,draw=black,fill=white] (4) at (0,2) {$z$};

\begin{scope}
    \path [-] (1) edge node {} (3);
    \path [-] (2) edge node {} (3);
    \path [dashed] (3) edge node {} (4);
\end{scope}
\end{tikzpicture}}}
\caption{Every $4$-subset which encompasses two or more vertices, all of which are in $T_1$}%
\label{fig:H4T1}%
\end{figure}
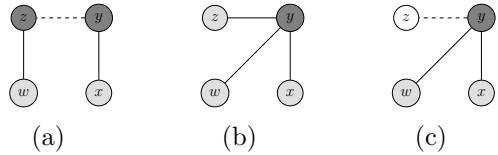

Now we need only to sum $f_G(H,V-W)-1$ for each such $4$-subset. We will sum each $f_G(H,V-W)-1$ by evaluating $f_G(H,V-W)-1$ for each case then multiplying it by the number of times it occurs in $G$.

If $H$ encompasses two leaves on different stems (Figure \ref{fig:H4T1} $(a)$) then $f_G(H,V-W)-1=1$. This $4$-subset will occur $\sum_{i=1}^{\omega} \frac{|S_i|}{2}(|T_1|-|S_i|)$ times. If $H$ encompasses three leaves on the same stem (Figure \ref{fig:H4T1} $(b)$) then $f_G(H,V-W)-1=2$. This $4$-subset will occur $\sum_{i=1}^{\omega} {|S_i| \choose 3}$ times. If $H$ encompasses two leaves on the same stem and another vertex which is not on that stem (Figure \ref{fig:H4T1} $(c)$) then $f_G(H,V-W)-1=1$. This $4$-subset will occur $\sum_{i=1}^{\omega} {|S_i| \choose 2}(n-|S_i|-1)$ times. Taking the sum of each of the cases gives us $\alpha_1$.

\vspace{5mm}
\noindent \textbf{Case 2:} \emph{$S \cap T_1 \neq \emptyset$ and $S \cap T_2 \neq \emptyset$}
\vspace{5mm}

\noindent As $H$ encompasses at least one vertex in $T_2$ and at least one vertex in $T_1$ then without loss of generality let $x$ be encompassed by $H$ and $x \in T_2$ where $N(x)=\{w,y\}$. As $x \in T_2$ then $x$ is not a stem. Therefore $w$ and $y$ are not leaves and hence not in $T_1$. As $H$ must encompass at least one vertex in $T_1$, $z \in T_1$. Without loss of generality let $N(z)=\{y\}$. Note that the vertices of $H$ are uniquely determined by the neighbourhoods of $x$ and $z$. Furthermore $y$ must be a stem and $w$ can either be a stem, in $T_2$ and encompassed by $H$, or in $T_2$ and not encompassed by $H$. Each case induces a subgraph shown in Figure \ref{fig:H4T12}. Dark gray vertices are stems, light gray vertices are encompassed by $H$ and dashed edges are possible edges.

\begin{figure}[!h]
\def\c{0.5}
\centering
\subfigure[]{
\scalebox{\c}{
\begin{tikzpicture}
\node[shape=circle,draw=black,fill=gray] (1) at (0,0) {$w$};
\node[shape=circle,draw=black,fill=gray!25] (2) at (2,0) {$x$};
\node[shape=circle,draw=black,fill=gray] (3) at (2,2) {$y$};
\node[shape=circle,draw=black,fill=gray!25] (4) at (0,2) {$z$};

\begin{scope}
    \path [-] (1) edge node {} (2);
    \path [-] (2) edge node {} (3);
    \path [-] (3) edge node {} (4);
    \path [dashed] (1) edge node {} (3);
\end{scope}
\end{tikzpicture}}}
\qquad
\subfigure[]{
\scalebox{\c}{
\begin{tikzpicture}
\node[shape=circle,draw=black,fill=gray!25] (1) at (0,0) {$w$};
\node[shape=circle,draw=black,fill=gray!25] (2) at (2,0) {$x$};
\node[shape=circle,draw=black,fill=gray] (3) at (2,2) {$y$};
\node[shape=circle,draw=black,fill=gray!25] (4) at (0,2) {$z$};

\begin{scope}
    \path [-] (1) edge node {} (2);
    \path [-] (2) edge node {} (3);
    \path [-] (3) edge node {} (4);
    \path [-] (1) edge node {} (3);
\end{scope}
\end{tikzpicture}}}
\qquad
\subfigure[]{
\scalebox{\c}{
\begin{tikzpicture}
\node[shape=circle,draw=black,fill=white] (1) at (0,0) {$w$};
\node[shape=circle,draw=black,fill=gray!25] (2) at (2,0) {$x$};
\node[shape=circle,draw=black,fill=gray] (3) at (2,2) {$y$};
\node[shape=circle,draw=black,fill=gray!25] (4) at (0,2) {$z$};

\begin{scope}
    \path [-] (1) edge node {} (2);
    \path [-] (2) edge node {} (3);
    \path [-] (3) edge node {} (4);
\end{scope}
\end{tikzpicture}}}
\caption{Every $4$-subset which encompasses at least one vertex from $T_1$ and at least one vertex from $T_2$}%
\label{fig:H4T12}%
\end{figure}
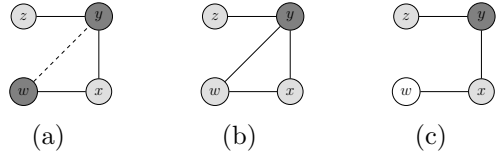

If $w$ is a stem (Figure \ref{fig:H4T12} $(a)$) then $f_G(H,V-W)-1=1$ and $x \in V_2^{ij}$ for some stems $i$ and $j$. This $4$-subset will occur $|S_i|+|S_j|$ times for every $T_2$ vertex adjacent to stems $i$ and $j$. As the number of $T_2$ vertices adjacent to stems $i$ and $j$ is $V_2^{ij}$ then this $4$-subset will occur $\sum\limits_{i \neq j} |V_2^{ij}|(|S_i|+|S_j|)$ times.

If $w$ is not a stem (Figure \ref{fig:H4T12} $(b)$ and $(c)$) then $x \in V_1^i$ and $z \in S_i$ for some stem $i$. As $H$ is uniquely determined by the closed neighbourhoods of $x$ and $z$ then we can count these by choosing one vertex from $V_1^i$ and one vertex from $S_i$ for each stem $i$. This gives us the term $\sum\limits_{i=1}^{\omega} |V_1^{i}||S_i|$, and by it, the subgraph of the type in Figure \ref{fig:H4T12} $(b)$ will be counted twice for each instance in $G$ and subgraph in Figure \ref{fig:H4T12} $(c)$ will be counted once for each instance in $G$. However, that is exactly equal to $f_G(H,V-W)-1$ for each of these cases. Hence $\sum\limits_{i=1}^{\omega} |V_1^{i}||S_i|$ is equal to $f_G(H,V-W)-1$ multiplied by the number of times it occurs in $G$.

Taking the sum of terms for when $w$ is a stem and when $w$ is not a stem gives us $\alpha_2$.

\vspace{5mm}
\noindent \textbf{Case 3:} \emph{$S \cap T_1 = \emptyset$ and $S \cap T_2 \neq \emptyset$}
\vspace{5mm}

\noindent We will generate every possible such subgraph by first constructing the induced subgraphs of $4$-subsets which encompass at least one degree two vertex. Clearly the smallest (fewest edges) such subgraph is $P_3 \cup K_1$ as shown in Figure \ref{fig:H4T2} $(a)$. We can then construct the other such subgraphs by adding every combination of the four omitted edges and removing any isomorphisms. This generates the seven other subgraphs shown in Figure \ref{fig:H4T2}.

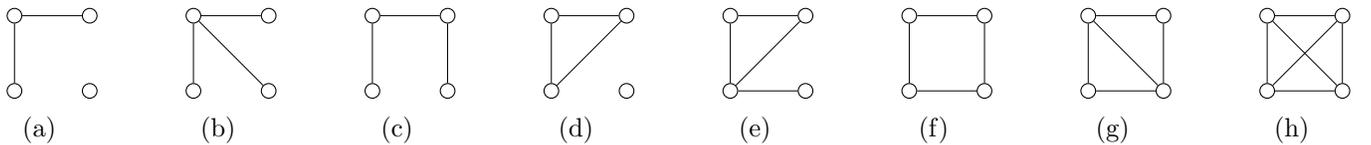
\begin{figure}[!h]
\def\c{0.5}
\centering
\subfigure[]{
\scalebox{\c}{
\begin{tikzpicture}
\begin{scope}[every node/.style={circle,thick,draw}]
    \node (1) at (0,0) {};
    \node (2) at (0,2) {};
    \node (3) at (2,2) {};
    \node (4) at (2,0) {};
\end{scope}

\begin{scope}
    \path [-] (1) edge node {} (2);
    \path [-] (2) edge node {} (3);
\end{scope}
\end{tikzpicture}}}
\qquad
\subfigure[]{
\scalebox{\c}{
 \begin{tikzpicture}
\begin{scope}[every node/.style={circle,thick,draw}]
    \node (1) at (0,0) {};
    \node (2) at (0,2) {};
    \node (3) at (2,2) {};
    \node (4) at (2,0) {};
\end{scope}

\begin{scope}
    \path [-] (1) edge node {} (2);
    \path [-] (2) edge node {} (3);
    \path [-] (2) edge node {} (4);
\end{scope}
\end{tikzpicture}}}
\qquad
\subfigure[]{
\scalebox{\c}{
\begin{tikzpicture}
\begin{scope}[every node/.style={circle,thick,draw}]
    \node (1) at (0,0) {};
    \node (2) at (0,2) {};
    \node (3) at (2,2) {};
    \node (4) at (2,0) {};
\end{scope}

\begin{scope}
    \path [-] (1) edge node {} (2);
    \path [-] (2) edge node {} (3);
    \path [-] (3) edge node {} (4);
\end{scope}
\end{tikzpicture}}}
\qquad
\subfigure[]{
\scalebox{\c}{
\begin{tikzpicture}
\begin{scope}[every node/.style={circle,thick,draw}]
    \node (1) at (0,0) {};
    \node (2) at (0,2) {};
    \node (3) at (2,2) {};
    \node (4) at (2,0) {};
\end{scope}

\begin{scope}
    \path [-] (1) edge node {} (2);
    \path [-] (2) edge node {} (3);
    \path [-] (3) edge node {} (1);
\end{scope}
\end{tikzpicture}}}
\qquad
\subfigure[]{
\scalebox{\c}{
\begin{tikzpicture}
\begin{scope}[every node/.style={circle,thick,draw}]
    \node (1) at (0,0) {};
    \node (2) at (0,2) {};
    \node (3) at (2,2) {};
    \node (4) at (2,0) {};
\end{scope}

\begin{scope}
    \path [-] (1) edge node {} (2);
    \path [-] (2) edge node {} (3);
    \path [-] (3) edge node {} (1);
    \path [-] (1) edge node {} (4);
\end{scope}
\end{tikzpicture}}}
\qquad
\subfigure[]{
\scalebox{\c}{
\begin{tikzpicture}
\begin{scope}[every node/.style={circle,thick,draw}]
    \node (1) at (0,0) {};
    \node (2) at (0,2) {};
    \node (3) at (2,2) {};
    \node (4) at (2,0) {};
\end{scope}

\begin{scope}
    \path [-] (1) edge node {} (2);
    \path [-] (2) edge node {} (3);
    \path [-] (3) edge node {} (4);
    \path [-] (4) edge node {} (1);
\end{scope}
\end{tikzpicture}}}
\qquad
\subfigure[]{
\scalebox{\c}{
\begin{tikzpicture}
\begin{scope}[every node/.style={circle,thick,draw}]
    \node (1) at (0,0) {};
    \node (2) at (0,2) {};
    \node (3) at (2,2) {};
    \node (4) at (2,0) {};
\end{scope}

\begin{scope}
    \path [-] (1) edge node {} (2);
    \path [-] (2) edge node {} (3);
    \path [-] (3) edge node {} (4);
    \path [-] (4) edge node {} (1);
    \path [-] (2) edge node {} (4);
\end{scope}
\end{tikzpicture}}}
\qquad
\subfigure[]{
\scalebox{\c}{
\begin{tikzpicture}
\begin{scope}[every node/.style={circle,thick,draw}]
    \node (1) at (0,0) {};
    \node (2) at (0,2) {};
    \node (3) at (2,2) {};
    \node (4) at (2,0) {};
\end{scope}

\begin{scope}
    \path [-] (1) edge node {} (2);
    \path [-] (2) edge node {} (3);
    \path [-] (3) edge node {} (4);
    \path [-] (4) edge node {} (1);
    \path [-] (2) edge node {} (4);
    \path [-] (1) edge node {} (3);
\end{scope}
\end{tikzpicture}}}
\caption{Every subgraph with four vertices containing at least one degree two vertex}%
\label{fig:H4T2}%
\end{figure}

We now narrow the subgraphs in Figure \ref{fig:H4T2} to subgraphs which encompass two or more vertices in $T_2$. Simply put, each subgraph must contain at least two degree two vertices which are not stems. As subgraph $(a)$ and $(b)$ only contain one vertex for degree greater than one, they do not fit our criteria. As these subgraphs are from a graph in $\mathcal{G}_2$, any vertex with degree greater than two must be a stem and hence not in $T_2$. As each vertex in subgraph $(h)$ is degree three then they are all stems and not in $T_2$. Therefore subgraph $(h)$ does not fit our criteria and we need only to consider subgraphs $(c)$, $(d)$, $(e)$, $(f)$, and $(g)$.

Of the remaining subgraphs we must consider the possibility that some degree two vertices are not in $T_2$ or are not encompassed by $H$. As each subgraph must contain at least two $T_2$ vertices, the degree two vertices in subgraphs $(c)$, $(e)$, and $(g)$ cannot be stems. Each case is shown in Figure \ref{fig:H4T2s} where stems are the dark gray vertices and the vertices in $T_2$ are in light gray.

\renewcommand\thesubfigure{(\roman{subfigure})}

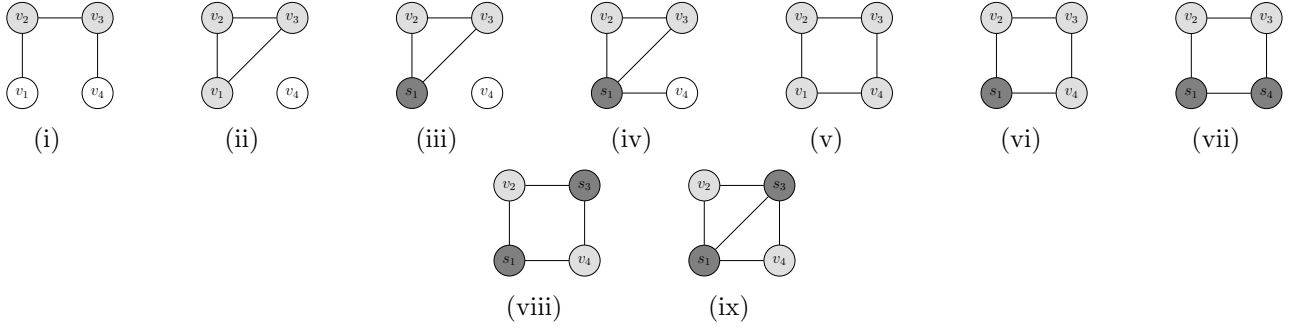
\begin{figure}[!h]
\def\c{0.50}
\centering
\subfigure[]{
\scalebox{\c}{
\begin{tikzpicture}
\node[shape=circle,draw=black,fill=white] (1) at (0,0) {$v_1$};
\node[shape=circle,draw=black,fill=gray!25] (2) at (0,2) {$v_2$};
\node[shape=circle,draw=black,fill=gray!25] (3) at (2,2) {$v_3$};
\node[shape=circle,draw=black,fill=white] (4) at (2,0) {$v_4$};

\begin{scope}
    \path [-] (1) edge node {} (2);
    \path [-] (2) edge node {} (3);
    \path [-] (3) edge node {} (4);
\end{scope}
\end{tikzpicture}}}
\qquad
\subfigure[]{
\scalebox{\c}{
 \begin{tikzpicture}
\node[shape=circle,draw=black,fill=gray!25] (1) at (0,0) {$v_1$};
\node[shape=circle,draw=black,fill=gray!25] (2) at (0,2) {$v_2$};
\node[shape=circle,draw=black,fill=gray!25] (3) at (2,2) {$v_3$};
\node[shape=circle,draw=black,fill=white] (4) at (2,0) {$v_4$};

\begin{scope}
    \path [-] (1) edge node {} (2);
    \path [-] (2) edge node {} (3);
    \path [-] (3) edge node {} (1);
\end{scope}
\end{tikzpicture}}}
\qquad
\subfigure[]{
\scalebox{\c}{
\begin{tikzpicture}
\node[shape=circle,draw=black,fill=gray] (1) at (0,0) {$s_1$};
\node[shape=circle,draw=black,fill=gray!25] (2) at (0,2) {$v_2$};
\node[shape=circle,draw=black,fill=gray!25] (3) at (2,2) {$v_3$};
\node[shape=circle,draw=black,fill=white] (4) at (2,0) {$v_4$};

\begin{scope}
    \path [-] (1) edge node {} (2);
    \path [-] (2) edge node {} (3);
    \path [-] (3) edge node {} (1);
\end{scope}
\end{tikzpicture}}}
\qquad
\subfigure[]{
\scalebox{\c}{
\begin{tikzpicture}
\node[shape=circle,draw=black,fill=gray] (1) at (0,0) {$s_1$};
\node[shape=circle,draw=black,fill=gray!25] (2) at (0,2) {$v_2$};
\node[shape=circle,draw=black,fill=gray!25] (3) at (2,2) {$v_3$};
\node[shape=circle,draw=black,fill=white] (4) at (2,0) {$v_4$};

\begin{scope}
    \path [-] (1) edge node {} (2);
    \path [-] (2) edge node {} (3);
    \path [-] (3) edge node {} (1);
    \path [-] (1) edge node {} (4);
\end{scope}
\end{tikzpicture}}}
\qquad
\subfigure[]{
\scalebox{\c}{
\begin{tikzpicture}
\node[shape=circle,draw=black,fill=gray!25] (1) at (0,0) {$v_1$};
\node[shape=circle,draw=black,fill=gray!25] (2) at (0,2) {$v_2$};
\node[shape=circle,draw=black,fill=gray!25] (3) at (2,2) {$v_3$};
\node[shape=circle,draw=black,fill=gray!25] (4) at (2,0) {$v_4$};

\begin{scope}
    \path [-] (1) edge node {} (2);
    \path [-] (2) edge node {} (3);
    \path [-] (3) edge node {} (4);
    \path [-] (4) edge node {} (1);
\end{scope}
\end{tikzpicture}}}
\qquad
\subfigure[]{
\scalebox{\c}{
\begin{tikzpicture}
\node[shape=circle,draw=black,fill=gray] (1) at (0,0) {$s_1$};
\node[shape=circle,draw=black,fill=gray!25] (2) at (0,2) {$v_2$};
\node[shape=circle,draw=black,fill=gray!25] (3) at (2,2) {$v_3$};
\node[shape=circle,draw=black,fill=gray!25] (4) at (2,0) {$v_4$};

\begin{scope}
    \path [-] (1) edge node {} (2);
    \path [-] (2) edge node {} (3);
    \path [-] (3) edge node {} (4);
    \path [-] (4) edge node {} (1);
\end{scope}
\end{tikzpicture}}}
\qquad
\subfigure[]{
\scalebox{\c}{
\begin{tikzpicture}
\node[shape=circle,draw=black,fill=gray] (1) at (0,0) {$s_1$};
\node[shape=circle,draw=black,fill=gray!25] (2) at (0,2) {$v_2$};
\node[shape=circle,draw=black,fill=gray!25] (3) at (2,2) {$v_3$};
\node[shape=circle,draw=black,fill=gray] (4) at (2,0) {$s_4$};

\begin{scope}
    \path [-] (1) edge node {} (2);
    \path [-] (2) edge node {} (3);
    \path [-] (3) edge node {} (4);
    \path [-] (4) edge node {} (1);
\end{scope}
\end{tikzpicture}}}
\qquad
\subfigure[]{
\scalebox{\c}{
\begin{tikzpicture}
\node[shape=circle,draw=black,fill=gray] (1) at (0,0) {$s_1$};
\node[shape=circle,draw=black,fill=gray!25] (2) at (0,2) {$v_2$};
\node[shape=circle,draw=black,fill=gray] (3) at (2,2) {$s_3$};
\node[shape=circle,draw=black,fill=gray!25] (4) at (2,0) {$v_4$};

\begin{scope}
    \path [-] (1) edge node {} (2);
    \path [-] (2) edge node {} (3);
    \path [-] (3) edge node {} (4);
    \path [-] (4) edge node {} (1);
\end{scope}
\end{tikzpicture}}}
\qquad
\subfigure[]{
\scalebox{\c}{
\begin{tikzpicture}
\node[shape=circle,draw=black,fill=gray] (1) at (0,0) {$s_1$};
\node[shape=circle,draw=black,fill=gray!25] (2) at (0,2) {$v_2$};
\node[shape=circle,draw=black,fill=gray] (3) at (2,2) {$s_3$};
\node[shape=circle,draw=black,fill=gray!25] (4) at (2,0) {$v_4$};

\begin{scope}
    \path [-] (1) edge node {} (2);
    \path [-] (2) edge node {} (3);
    \path [-] (3) edge node {} (4);
    \path [-] (4) edge node {} (1);
    \path [-] (1) edge node {} (3);
\end{scope}
\end{tikzpicture}}}
\caption{Every subgraph with four vertices containing two or more vertices in $T_2$}%
\label{fig:H4T2s}%
\end{figure}

\renewcommand\thesubfigure{(\alph{subfigure})}

Note that in Figure \ref{fig:H4T2s} the white vertices are not encompassed and can either be stems, $T_1$, or $T_2$ vertices. However as we are examining the case where $H$ only encompasses $T_2$ vertices then $v_1$, $v_4$ from $(i)$ and $v_4$ from $(iv)$ are not $T_1$ vertices.

Now we need only to sum $f_G(H,V-W)-1$ for each $4$-subset. We will sum each $f_G(H,V-W)-1$ by evaluating $f_G(H,V-W)-1$ for each case, then multiplying the result by the number of times the subgraph occurs in $G$. We may also group some cases for simplicity.

Cases $(i)-(vii)$ from Figure \ref{fig:H4T2s} all encompass the adjacent $T_2$ vertices $v_2$ and $v_3$. Furthermore $N[v_2] \neq N[v_3]$ in the cases $(i)$, $(v)$, $(vi)$, and $(vii)$. Therefore $|N[v_2] \cup N[v_3]|=4$ and any subset would require four vertices to encompass both $v_2$ and $v_3$. Therefore there is exactly one $4$-subset of $G$ which encompasses $v_2$ and $v_3$. As there is exactly one edge between $v_2$ and $v_3$ in $G$, we can relate the number of edges between $T_2$ vertices in $G$ and the sum of $f_G(H,V-W)-1$ for cases $(i)$, $(v)$, $(vi)$, and $(vii)$. We will count the total number edges between $T_2$ vertices in $G$ by summing half the number of $T_2$ vertices each $T_2$ vertex is adjacent to. We will then subtract the number of edges between $T_2$ vertices which have the same closed neighbourhood as there are multiple 4-subsets which encompass them. We will also adjust for cases where the number of edges between $T_2$ vertices does not equal $f_G(H,V-W)-1$.

For a graph in $\mathcal{G}_2$, the neighbours of a $T_2$ vertex are either stems or other $T_2$ vertices. Each vertex in $V_0$ is adjacent to two other $T_2$ vertices. Each vertex in $V_1^i$, for any stem $i$, is adjacent to one other $T_2$ vertex. Each vertex in $V_2^{ij}$, for any stems $i$ and $j$, is adjacent to no other $T_2$ vertices. Therefore the number of edges between $T_2$ vertices in $G$ is 

$$ \frac{1}{2}\bigg(2|V_0|+\sum\limits_{i=1}^{\omega} |V_1^i|\bigg).$$

\noindent If two adjacent $T_2$ vertices have the same closed neighbourhood then they induce a $3$-cycle. Furthermore, as at least two of the vertices of the $3$-cycle are in $T_2$, the induced $3$-cycle is either a $3$-loop or $3$-cycle component in $G$. As each $3$-loop contains one edge between $T_2$ vertices and each $3$-cycle component contains three edges between $T_2$ vertices, we subtract $|\mathcal{L}_3|+3|\mathcal{C}_3|$ from the total number of edges between $T_2$ vertices.

In cases $(i)$, $(vi)$, and $(vii)$ the number of edges between $T_2$ vertices equals $f_G(H,V-W)-1$. However, in case $(v)$, which is a $C_4$ component of $G$, $f_G(H,V-W)-1=3$ and there are 4 edges between $T_2$ vertices. Hence we must also subtract one for each $C_4$ component of $G$. Therefore the sum of $f_G(H,V-W)-1$ for cases $(i)$, $(vi)$, $(v)$, and $(vii)$ is

$$ \frac{1}{2}\bigg(2|V_0|+\sum\limits_{i=1}^{\omega} |V_1^i|\bigg)-|\mathcal{L}_3|-3|\mathcal{C}_3|-|\mathcal{C}_4|.$$

For case $(ii)$, $f_G(H,V-W)-1 = 2$. Case $(ii)$ is a $C_3$ component with any other vertex. Thus for each $C_3$ component there is $n-3$ such $4$-subsets. Therefore the number of instances of case $(ii)$ is $|\mathcal{C}_3|(n-3)$.

For cases $(iii)$ and $(iv)$, $f_G(H,V-W)-1 = 1$. Cases $(iii)$ and $(iv)$ are $3$-loops in $G$ with any other vertex which is not a $T_1$ adjacent to the stem. This is true because $H$ does not encompass any $T_1$ vertices and hence cannot contain both a stem and one of its leaves. The number of instances of the cases $(iii)$ and $(iv)$ in $G$ is $\sum_{i=1}^{\omega}|\mathcal{L}_3^i|(n-3-|S^i|)$.

For cases $(viii)$ and $(ix)$, $f_G(H,V-W)-1 = 1$. Cases $(viii)$ and $(ix)$ are two $T_2$ vertices adjacent to the same two stems. Therefore for each pair of stems $i$ and $j$, there are ${|V_2^{ij}| \choose 2}$ such $4$-subsets. Therefore the number of instances of cases $(viii)$ and $(ix)$ is $\sum_{i \neq j} {|V_2^{ij}| \choose 2}$.

The sum of $f_G(H,V-W)-1 = 1$ for cases $(i)-(ix)$ yields

$$ \frac{1}{2}\bigg(2|V_0|+\sum\limits_{i=1}^{\omega} |V_1^i|\bigg)+\sum_{i \neq j} {|V_2^{ij}| \choose 2}-|\mathcal{L}_3|-3|\mathcal{C}_3|-|\mathcal{C}_4|+2|\mathcal{C}_3|(n-3)+\sum_{i=1}^{\omega}|\mathcal{L}_3^i|(n-3-|S_i|).$$

\noindent Combining like terms and the fact $|\mathcal{L}_3| = \sum_{i=1}^{\omega}|\mathcal{L}_3^i|$ gives us $\alpha_3,$

$$|V_0|+\sum\limits_{i=1}^{\omega} \frac{|V_1^i|}{2}+\sum_{i \neq j} {|V_2^{ij}| \choose 2}-|\mathcal{C}_4|+|\mathcal{C}_3|(2n-9)+\sum_{i=1}^{\omega}|\mathcal{L}_3^i|(n-4-|S_i|).$$

\end{proof}
We now turn specifically to the coefficients of the domination polynomials of paths, with interest in the top four. 

\begin{theorem}
\label{thm:FamPathFacts}
\textnormal{\cite{2009Paths}}

\renewcommand{\labelenumi}{(\roman{enumi})}
 \begin{enumerate}
   \item For every $n \geq 2$, $d(P_n, n-1) = n$.
   \item For every $n \geq 3$, $d(P_n, n-2) = {n \choose 2} - 2$.
   \item For every $n \geq 4$, $d(P_n, n-3) = {n \choose 3} - (3n-8)$.
   \item For every $n \geq 5$, $d(P_n, n-4) = {n \choose 4} - (2n^2 -13n+20)$.
 \end{enumerate}
\QED
\end{theorem}


We will also need to know when $-2$ is root, as this plays a role in our characterization of graphs that are domination equivalent to paths. 
The domination polynomial is multiplicative across components (that is, $D(G_1 \cup \cdots G_m,x) = D(G_1,x)D(G_2,x)  \cdots D(G_m,x)$) and $D(K_2,-2)=0$. Therefore, $D(G,-2) \neq 0$ implies $G$ has no $K_2$ components. This observation is vital to proving the domination equivalence classes of paths.

It is well known (and easy to see \cite{2014Intro}) that $P_n$ satisfies the recurrence
$$D(G_n,x) = x(D(G_{n-1},x)+D(G_{n-2},x)+D(G_{n-3},x))$$
for $n \geq 3$ (other families, such as the cycles $C_{n}$, do as well). 
We show that for $n \geq 9$, $-2$ is {\em never} a root of $D(P_m,x)$, by showing that, given a sequence of graphs satisfying such a recurrence if the magnitude of $D(G_i,-2)$ is non-zero, increasing and of alternating sign for the four consecutive integers $i=N, N+1, N+2, N+3$, then $D(G_m,-2) \neq 0$ for $m \geq N$. This allows us to show that any $G \sim_{\mathcal{D}} G_m$ does not have any $K_2$ components, since $D(K_2,-2) = 0$, if $G$ had a $K_2$ component then $D(G,-2) = 0$ as well.

\begin{lemma}
\label{lem:gnminus2neq0}

Fix $k \geq 1$. Suppose we have a sequence of graphs $(G_n)_{n \geq 1}$ that satisfies the recurrence
$$D(G_n,x) = x(D(G_{n-1},x)+D(G_{n-2},x)+D(G_{n-3},x))$$
for $n \geq 3$.
If for some $N \in \mathbb{N}$
$$0<|D(G_{N},-2)| < |D(G_{N+1},-2)| < |D(G_{N+2},-2)| < |D(G_{N+3},-2)|$$
and $D(G_{N},-2)$, $D(G_{N+1},-2)$, $D(G_{N+2},-2)$, $D(G_{N+3},-2)$ have alternating sign, then $D(G_m,-2) \neq 0$ for $m \geq N$.
\end{lemma}

\begin{proof}
Substituting $x=-2$ into the recurrence , we find that
$$D(G_n,-2) = -2(D(G_{n-1},-2)+D(G_{n-2},-2)+D(G_{n-3},-2)).$$
By induction we will show the magnitude of $D(G_n,-2)$ is increasing in absolute value and alternating in sign for all $n \geq N+3$. As $|D(G_{N},-2)|>0$ then this will imply $D(G_n,-2) \neq 0$ for $m \geq N$. By the hypotheses, the result is true for $n = N$.

Suppose for some $k \geq N$, $D(G_{N+3},-2), \ldots, D(G_{k},-2)$ alternate in signs and increase in absolute value. Then we will first show $D(G_{k+1},-2)$ has opposite sign to $D(G_{k},-2)$. First assume $D(G_{k},-2)>0$ (a similar argument holds when $D(G_{k},-2)<0$). Then $D(G_{k-1},-2) < 0$ and $D(G_{k-2},-2) > 0$. By our induction assumption, the magnitude $D(G_{m},-2)$ is strictly increasing for $N+3 \leq m \leq k$. Therefore

$$ D(G_{k},-2) + D(G_{k-1},-2) + D(G_{k-2},-2) > 0$$.

\noindent When we multiply the left side of the above inequality by $-2$, from the recurrence relation for $D(G_k,-2)$ we will obtain $D(G_{k+1},-2)$. The signs continue to alternate.

We now show $|D(G_{k+1},-2)| > |D(G_{k},-2)|$. We consider the two cases: \\ $D(G_{k},-2)>0$ and $D(G_{k},-2)<0$.

If $D(G_{k},-2)>0$ then $D(G_{k-1},-2) < 0$, $D(G_{k-2},-2) > 0$, and $D(G_{k-3},-2) < 0$. By our induction assumption, the magnitude $D(G_{m},-2)$ is strictly increasing. Therefore

$$ D(G_{k-1},-2) + D(G_{k-2},-2) + D(G_{k-3},-2) < D(G_{k-1},-2) + D(G_{k-2},-2) < 0$$.

\noindent By the recurrence relation for $D(G_k,-2)$ we deduce

\begin{center}
\begin{tabular}{ l c l }
 $D(G_{k},-2)$ & $=$ & $-2(D(G_{k-1},-2)+D(G_{k-2},-2)+D(G_{k-3},-2))$ \\ 
 				 & $>$ & $-2(D(G_{k-1},-2)+D(G_{k-2},-2))$ \\
\end{tabular}
\end{center}

\noindent As $D(G_{k+1},-2)<0$ then

\begin{center}
\begin{tabular}{ l c l }
$|D(G_{k+1},-2)|$ & $=$ & $-D(G_{k+1},-2)$ \\
			     & $=$ & $-(-2(D(G_{k},-2)+D(G_{k-1},-2)+D(G_{k-2},-2)))$ \\ 
 				 & $=$ & $2D(G_{k},-2)+2D(G_{k-1},-2)+2D(G_{k-2},-2)$ \\
 				 & $>$ & $D(G_{k},-2)-2(D(G_{k-1},-2)+D(G_{k-2},-2))$ \\
 				 &     & $+2D(G_{k-1},-2)+2D(G_{k-2},-2)$ \\
 				 & $=$ & $D(G_{k},-2)$ \\
 				 & $=$ & $|D(G_{k},-2)|$ \\
\end{tabular}
\end{center}
Therefore $|D(G_{k+1},-2)|> |D(G_{k},-2)|$ and our claim is true. A similar argument holds when $D(G_k,-2)<0$.
\end{proof}

Using the base cases $D(P_1,x) = x$, $D(P_1,x) = x^2+2x$ and $D(P_3,x) = x^3+3x^2+x$, calculations will show that $D(P_{i},-2) \neq 0$ for $9 \leq i \leq 12$, $D(P_{13},-2) = -32$, $D(P_{14},-2) = 64$ and $D(P_{15},-2) = -96$. From this and the previous Lemma, we conclude:

\begin{corollary}
\label{pnminus2neq0}
If $n \geq 9$, $-2$ is not a zero of $D(P_n,x)$.
\QED
\end{corollary}

\section{Equivalence Classes of Paths}
\label{sec:EqPath}

We have done the necessary background work to proceed onto our characterization of those graphs that are domination equivalent to path $P_{n}$. We remark that are proof is considerably more involved  

%
We first observe that any graph $G \sim_{\mathcal{D}} P_n$ does not have any $4$-cycle components. This follows from the multiplicativity of the domination polynomial over components and the following two lemma.

%
%

\begin{lemma}
\label{lem:Cn-1}

\textnormal{\cite{2014EqCycleAll}} If $n$ is a positive integer, then

$$
D(C_n,-1) = \left\{
        \begin{array}{rl}
            3 & \quad n \equiv 0 $ mod $4 \\
            -1 & \quad $otherwise$ \\
        \end{array}
    \right.
$$
\QED
\end{lemma}

\begin{lemma}
\label{lem:FamForestn1}

\textnormal{\cite{2013n1}} Let $F$ be a forest. Then $D(F,-1) \in \{1,-1\}$ and therefore $D(P_n,-1) \in \{1,-1\}$.
\QED
\end{lemma}

\begin{corollary}
\label{cor:PathNo4cycle}
If a $G$ is $\mathcal{D}$-equivalent to $P_n$ with a component $H$, then $|D(H,-1)| =1$, and so $G$ does not have any 4-cycle components.
\QED
\end{corollary}

%
%

In the next Lemma we use the results from Theorem \ref{thm:d(G,n-3)}, Theorem \ref{thm:d(G,n-4)} and Corollary~\ref{pnminus2neq0} to show, for large enough $n$, any graph $G \sim_{\mathcal{D}} P_n$ must be the disjoint union of one path and some number of cycles.

\begin{lemma}
\label{lem:EqPUCk}
For $n \geq 9$, if $G \sim_{\mathcal{D}} P_n$ then $G = H \cup C$ where $H \in \{P_{k},\widetilde{P_{k}}\}$ and $C$ is a disjoint union of cycles.
\end{lemma}

\begin{proof}

Let $G$ be a graph with $D(G,x) = D(P_n,x)$ where $n \geq 9$. Then $d(G, i)=d(P_n,i)$ for all $i$. 
%
%
Furthermore, by Theorem \ref{thm:FamPathFacts} we have
\renewcommand{\labelenumi}{(\roman{enumi})}
 \begin{enumerate}
   \item $d(G, n-1) = n$.
   \item $d(G, n-2) = {n \choose 2} - 2$.
   \item $d(G, n-3) = {n \choose 3} - (3n-8)$.
   \item $d(G, n-4) = {n \choose 4} - (2n^2 -13n+20)$.
 \end{enumerate}

\noindent By Theorem \ref{thm:DomPolyGeo} the number of isolated vertices in $G$ is $n - d(G,n-1)=0$. By Corollary~\ref{pnminus2neq0}, $D(G,-2) \neq 0$ and again by Theorem \ref{thm:DomPolyGeo} the number of leaves is $|T_1| = {n \choose 2} - d(G,n-2) = 2$. By Theorem \ref{thm:d(G,n-3)}, as $G$ has no $K_2$ components and no isolated vertices,

$$d(G,n-3) = {n \choose 3} -\bigg(|T_1| \cdot (n-2)+|T_2|-\sum_{i=1}^{\omega} {|S_i| \choose 2}-|\mathcal{L}_3|-2|\mathcal{C}_3|\bigg).$$ 

\noindent Furthermore, from $|T_1|=2$ and $(iii)$ we know

$$n-4 = |T_2|-\sum_{i=1}^{\omega} {|S_i| \choose 2}-|\mathcal{L}_3|-2|\mathcal{C}_3|.$$

\noindent By rearranging for $|T_2|$ we get

$$|T_2| = n-4+\sum_{i=1}^{\omega} {|S_i| \choose 2}+|\mathcal{L}_3|+2|\mathcal{C}_3|.$$

\noindent We claim for $G$, $|\mathcal{L}_3|=0$, $|\mathcal{C}_3| = 0$ and $G \in \mathcal{G}_2$ (recall that $\mathcal{G}_2$ is the set of all graphs with maximum non-stem degree two, and that $\omega$ is the number of stems in $G$). We will show our claim is true using the fact that $n = \omega + \sum_{i \in \mathbb{N}} |T_i|$ so $n \geq \omega + |T_1|+|T_2|$. As $|T_1|=2$ then $T_2 \leq n - (2 + \omega)$. Also, if $n = \omega + |T_1|+|T_2|$ then $G \in \mathcal{G}_2$. As $G$ has two leaves, it either has one or two stems. We now consider the two cases for $G$. 

\vspace{5mm}
\noindent \textbf{Case 1:} \emph{$G$ has one stem.}
\vspace{5mm}

\noindent Then $\omega = 1$, $|S_1| = 2$, and $|T_2| \leq n-3$. Thus

$$|T_2| = n-4+{2 \choose 2}+|\mathcal{L}_3|+2|\mathcal{C}_3|.$$

\noindent As $|\mathcal{L}_3|+2|\mathcal{C}_3| \geq 0$ then $|T_2| \geq n-3$ and therefore $|T_2| = n-3$. Furthermore $|\mathcal{L}_3|+2|\mathcal{C}_3| = 0$ so $|\mathcal{L}_3|=0$ and $|\mathcal{C}_3| = 0$. As $\omega + |T_1|+|T_2| = n$, $G \in \mathcal{G}_2$.

\vspace{5mm}
\noindent \textbf{Case 2:} \emph{$G$ has two stems.}
\vspace{5mm}

\noindent Then $\omega = 2$ ,$|S_1| = 1$,$|S_2| = 1$, and $|T_2| \leq n-4$. Thus

$$|T_2| = n-4+0+|\mathcal{L}_3|+2|\mathcal{C}_3|.$$

\noindent As $|\mathcal{L}_3|+2|\mathcal{C}_3| \geq 0$ then $|T_2| \geq n-4$ and therefore $|T_2| = n-4$. Furthermore $|\mathcal{L}_3|+2|\mathcal{C}_3| = 0$ so $|\mathcal{L}_3|=0$ and $|\mathcal{C}_3| = 0$. As $\omega + |T_1|+|T_2| = n$, $G \in \mathcal{G}_2$.

\vspace{5mm}

For a graph in $\mathcal{G}_2$, a $T_2$ vertex can only be adjacent to stems or other $T_2$ vertices. Therefore the $T_2$ vertices form paths between stems, $r$-loops, and disjoint cycles in $\mathcal{G}_2$ graphs. As $G \in \mathcal{G}_2$, $G$ will be the disjoint union of some number of cycles and a subgraph $H$ which has one of the two forms shown in Figure \ref{fig:Pr1s2s}. (These two forms were noted for graphs domination equivalent to paths in \cite{2010Char}, but we shall need more than was used there about the types of subgraphs present to limit the possibilities). 

\begin{figure}[!h]
\def\c{0.5}
\def\d{0.75}
\def\e{0.75}
\centering
\subfigure[One stem]{
\scalebox{\c}{
\begin{tikzpicture}

\draw (0,1*\e) circle (1*\e);
\draw (0,1.5*\e) circle (1.5*\e);
\draw (0,2.5*\e) circle (2.5*\e);
\draw (0,3*\e) circle (3*\e);

\path (0,3*\e) -- node[auto=false]{\vdots} (0,5*\e);

\node[shape=circle,draw=black,fill=white] (1) at (0,0) {};
\node[shape=circle,draw=black,fill=white] (2) at (-1,-1) {};
\node[shape=circle,draw=black,fill=white] (3) at (1,-1) {};

\begin{scope}
    \path [-] (1) edge node {} (2);
    \path [-] (1) edge node {} (3);
\end{scope}
\end{tikzpicture}}}
\quad
\subfigure[Two stems]{
\scalebox{\c}{
\begin{tikzpicture}

\draw (-2-1*\d,0) circle (1*\d);
\draw (-2-1.5*\d,0) circle (1.5*\d);
\draw (-2-2.5*\d,0) circle (2.5*\d);
\draw (-2-3*\d,0) circle (3*\d);
\path (-2-3*\d,0) -- node[auto=false]{\ldots} (-2-5*\d,0);

\draw (2+1*\d,0) circle (1*\d);
\draw (2+1.5*\d,0) circle (1.5*\d);
\draw (2+2.5*\d,0) circle (2.5*\d);
\draw (2+3*\d,0) circle (3*\d);
\path (2+3*\d,0) -- node[auto=false]{\ldots} (2+5*\d,0);

\draw (0,0) ellipse (2 and 0.5);
\draw (0,0) ellipse (2 and 1);
\draw (0,0) ellipse (2 and 1.5);

\path (0,0.5) -- node[auto=false]{\vdots} (0,-0.5);



\node[shape=circle,draw=black,fill=white] (1) at (-2,0) {};
\node[shape=circle,draw=black,fill=white] (2) at (-2,-3) {};
\node[shape=circle,draw=black,fill=white] (3) at (2,0) {};
\node[shape=circle,draw=black,fill=white] (4) at (2,-3) {};

\begin{scope}
    \path [-] (1) edge node {} (2);
    \path [-] (3) edge node {} (4);
\end{scope}
\end{tikzpicture}}}
\caption{The two possible structures of $H$}%
\label{fig:Pr1s2s}%
\end{figure}
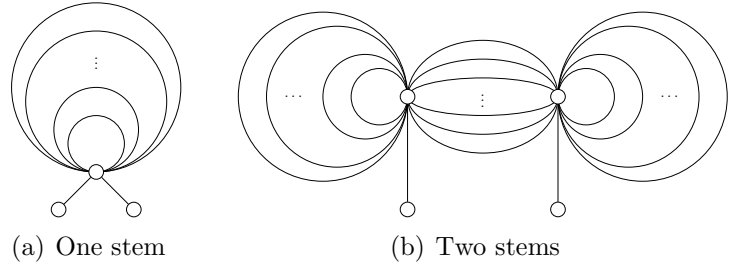

Recall from Section \ref{sec:Coeff}, we partitioned $T_2$ into subsets based on the number of neighbouring stems.

\begin{enumerate}
\item[$\bullet$] $V_0$: The subset of $T_2$ with no adjacent stems.
\item[$\bullet$] $V_1^{i}$:  The subset of $T_2$ adjacent to exactly one stem, stem $i$.
\item[$\bullet$] $V_2^{ij} $:  The subset of $T_2$ adjacent to exactly two stems, stems $i$ and $j$ (\emph{denoted $V_2$ when $G$ only has two stems} ).
\end{enumerate}

We wish to show that the subgraph $H$ of $G$ is either a path or a path with an edge between its stems. This is equivalent of showing $H$ has two stems with either one path between them or two, with one being an edge, and no $r$-loops. If $G$ has exactly two stems, and no $r$-loops, then the number of paths between the stems is exactly $\frac{1}{2}(|V_1^1|+|V_1^2|)+|V_2|$. Furthermore, if $|V_1^1| \leq 1$ and $|V_1^2| \leq 1$ then $H$ has no $r$-loops. Therefore it is sufficient to show $H$ has two stems and either $|V_1^1|=|V_1^2|=0$ and $|V_2|=1$, or $|V_1^1|=|V_1^2|=1$ and $|V_2|=0$. We will show this by examining $d(G,n-4)$.

By Theorem \ref{thm:d(G,n-4)}, as $G$ has no $K_2$ components, no isolated vertices, and $G \in \mathcal{G}_2$, we have that

$$d(G,n-4) = {n \choose 4} -\bigg(|T_1| {n-2\choose 2} + |T_2|(n-3) - \alpha_1 -\alpha_2-\alpha_3     \bigg)$$

\noindent where

\begin{center}
\begin{tabular}{ l c l }
$\alpha_1$ & $=$ & $\sum\limits_{i=1}^{\omega} {|S_i| \choose 2}(n-|S_i|-1)+\sum\limits_{i=1}^{\omega} \frac{|S_i|}{2}(|T_1|-|S_i|)+2\sum\limits_{i=1}^{\omega} {|S_i| \choose 3}$ \\ 

$\alpha_2$ & $=$ & $\sum\limits_{i=1}^{\omega} |V_1^{i}||S_i|+\sum\limits_{i \neq j} |V_2^{ij}|(|S_i|+|S_j|)$ \\ 

$\alpha_3$ & $=$ & $|V_0|+\sum\limits_{i=1}^{\omega} \frac{|V_1^i|}{2}+\sum\limits_{i \neq j} {|V_2^{ij}| \choose 2}-|\mathcal{C}_4|+|\mathcal{C}_3|(2n-9)+\sum\limits_{i=1}^{\omega}|\mathcal{L}_3^i|(n-4-|S_i|)$ \\ 
\end{tabular}
\end{center}

\noindent As $|\mathcal{L}_3|=0$ then $|\mathcal{L}_3^i|=0$ for every $i$. Furthermore $|\mathcal{C}_3|=0$ and by Corollary \ref{cor:PathNo4cycle} $|\mathcal{C}_4|=0$. We again consider the two cases where $G$ has one stem and $G$ has two stems. Note that $|T_2|=|V_0|+\sum_{i=1}^{\omega} |V_1^{i}|+\sum_{i \neq j} |V_2^{ij}|$. 


\vspace{5mm}
\noindent \textbf{Case 1:} \emph{$G$ has one stem.}
\vspace{5mm}

\noindent We claim this case results in a contradiction. As $G$ has one stem then $\omega = 1$, $|S_1| = 2$, $|T_2| = n-3$. As $G$ only has one stem, there are no degree two vertices adjacent to two stems and $|V_2^{ij}|=0$ for all $i$ and $j$. Furthermore $|V_0|+|V_1^1| = |T_2| = n-3$. Using this we can simplify $\alpha_1$, $\alpha_2$, and $\alpha_3$ to be

\begin{center}
\begin{tabular}{ l c l }
$\alpha_1$ & $=$ & ${2 \choose 2}(n-2-1)+\frac{2}{2}(2-2)+2 \cdot 0 = n-3.$ \\ 

$\alpha_2$ & $=$ & $2|V_1^{1}|.$ \\ 

$\alpha_3$ & $=$ & $|V_0|+\frac{|V_1^1|}{2}.$ \\ 
\end{tabular}
\end{center}

\noindent As $|V_0|+|V_1^1| = n-3$ then $\alpha_1+\alpha_2+\alpha_3 = 2n-6+\frac{3|V_1^1|}{2}$ and

$$d(G,n-4) = {n \choose 4} -\bigg(2n^2-13n+21 - \frac{3|V_1^1|}{2}\bigg)$$

\noindent However by item $(iv)$, $d(G,n-4) = {n \choose 4} -(2n^2-13n+20)$ and therefore $|V_1^1|=\frac{2}{3}$. But $|V_1^1|$ is a positive integer, which gives us a contradiction.

\vspace{5mm}
\noindent \textbf{Case 2:} \emph{$G$ has two stems.}
\vspace{5mm}

\noindent As $G$ has two stems, we find that $\omega = 2$ ,$|S_1| = 1$, $|S_2| = 1$, and $|T_2| = n-4$. As there are only two stems, let the set of $T_2$ vertices which are adjacent to both be denoted $V_2$. Furthermore $|V_0|+|V_1^1|+|V_1^2|+|V_2| = |T_2| = n-4$. Using this we can simplify $\alpha_1$, $\alpha_2$, and $\alpha_3$ to be

\begin{center}
\begin{tabular}{ l c l }
$\alpha_1$ & $=$ & $0+\sum\limits_{i=1}^{2} \frac{1}{2}(2-1)=1$ \\ 

$\alpha_2$ & $=$ & $\sum\limits_{i=1}^{2} |V_1^{i}|+2|V_2|$ \\ 

$\alpha_3$ & $=$ & $|V_0|+\sum\limits_{i=1}^{2} \frac{|V_1^i|}{2}+{|V_2| \choose 2}$ \\ 
\end{tabular}
\end{center}

\noindent As $|V_0|+|V_1^1|+|V_1^2|+|V_2| = n-4$ then $\alpha_1+\alpha_2+\alpha_3 = n-3 +\sum\limits_{i=1}^{2} \frac{|V_1^i|}{2} +|V_2|+{|V_2| \choose 2}$ and

$$d(G,n-4) = {n \choose 4} -\bigg(2n^2-13n+21-\sum\limits_{i=1}^{2} \frac{|V_1^i|}{2}- |V_2|-{|V_2| \choose 2}    \bigg)$$

\noindent However by item $(iv)$, $d(G,n-4) = {n \choose 4} -(2n^2-13n+20)$ and therefore

$$\sum_{i=1}^{2} \frac{|V_1^i|}{2}+ |V_2|+{|V_2| \choose 2}=1.$$

\noindent As each summand is non-negative and $|V_2|+{|V_2| \choose 2}$ is a non-negative integer, the only solutions to this are $\sum_{i=1}^{2} \frac{|V_1^i|}{2}=1,|V_2|=0$ or $\sum_{i=1}^{2} \frac{|V_1^i|}{2}=0,|V_2|=1$. 

In the case $\sum_{i=1}^{2} \frac{|V_1^i|}{2}=1,|V_2|=0$, then as $G$ has no $K_2$ components and there are no vertices adjacent to both stems ($|V_2|=0$) then $|V_1^i| \geq 1$ for each $i$. Furthermore as $\sum_{i=1}^{2} \frac{|V_1^i|}{2}=1$, then $|V_1^1|=|V_1^2|=1$, and $|V_2|=0$. In the case  $\sum_{i=1}^{2} \frac{|V_1^i|}{2}=0,|V_2|=1$ as $|V_1^i| \geq 0$ for each $i$ then $|V_1^1|=|V_1^2|=0$, and $|V_2|=1$. Both cases result in $H$ having one path between its two stems and no $r$-loops. As we do not specify the degree of the stems, this allows for the possibility of an edge to be between them, and proves our result.

\end{proof}

Let $n \in \mathbb{Z}$ and $p$ be a prime number. If $n$ is not zero, there is a nonnegative integer $a$ such that $p^a \mid n$ but $p^{a+1} \nmid n$; we let $ord_p(n) = a$. In other words, $a$ is the exponent of $p$ in the prime decomposition of $n$.  Furthermore let $ord_p(0) = 0$. In a similar method used by Akbari and Oboudi \cite{2014EqCycleAll} we will determine $ord_3(D(P_n,-3))$ in order to show that if a graph $G$ is $\mathcal{D}$-equivalent to a path, then $G$ is the disjoint union of a path and at most two cycles.

\begin{lemma} \cite{2014EqCycleAll} For $n \in \mathbb{N}$
\label{lem:Cn_ord3}

\vspace{5mm}
$ord_3(D(C_n,-3)) = \left\{
        \begin{array}{ll}
            \lceil \frac{n}{3} \rceil+1                                  & \quad n \equiv 0 $ mod $3 \\
           \lceil \frac{n}{3} \rceil $ or $ \lceil \frac{n}{3} \rceil +1 & \quad n \equiv 1 $ mod $3 \\
           \lceil \frac{n}{3} \rceil                                     & \quad n \equiv 2 $ mod $3 \\
        \end{array}
    \right.$

\QED
\end{lemma}

Using a similar approach for paths, we can prove.

\begin{lemma} For $n \in \mathbb{N}$
\label{lem:Pn_ord3}

\vspace{5mm}
$ord_3(D(P_n,-3)) = \left\{
        \begin{array}{ll}
            \lceil \frac{n}{3} \rceil                                    & \quad n \equiv 0 $ mod $3 \\
           \lceil \frac{n}{3} \rceil                                     & \quad n \equiv 1 $ mod $3 \\
           \lceil \frac{n}{3} \rceil \text{ or } \lceil \frac{n}{3} \rceil +1 & \quad n \equiv 2 $ mod $3 \\
        \end{array}
    \right.$
    
    \QED

\end{lemma}

The next straightforward lemma gives the domination numbers of paths and cycles, and will help us to restrict the number of disjoint cycles in $G$ if $G \sim_{\mathcal{D}} P_n$.

\begin{lemma}
\label{lem:FamPCMin}
For every $n \geq 1$, $\gamma(P_n) = \lceil \frac{n}{3} \rceil$, and for all $n \geq 3$, $\gamma(C_n) = \lceil \frac{n}{3} \rceil$.
\QED
\end{lemma}

From Lemma \ref{lem:EqPUCk} we know if $G$ is $\mathcal{D}$-equivalent to $P_n$ then $G$ is the disjoint union of $H$ and some number of cycles where $H \in \{P_{k},\widetilde{P_{k}}\}$ and $k \leq n$. In the next lemma we will show the number of cycles is at most two.

\begin{lemma} For $n \in \mathbb{N}$
\label{lem:EqPCC}
For $n \geq 9$, if $G \sim_{\mathcal{D}} P_n$ then $G = H \cup C$ where $H \in \{P_{k},\widetilde{P_{k}}\}$, $k \leq n$, and $C$ is a disjoint union of at most two cycles.
\end{lemma}  

\begin{proof}
Let $n=3m+r$ and $G$ be a graph with $D(G,x) = D(P_{3m+r},x)$ where $r \in \{0,1,2\}$. By Lemma \ref{lem:EqPUCk},

 $$G = P_{3m_1+r_1} \cup C_{3m_2+r_2} \cup \ldots \cup C_{3m_k+r_k},$$
 
\noindent  where $3m+r = \sum_{i=1}^k (3m_i + r_i)$ and for each $i$, $r_i \in \{0,1,2\}$. In this proof we will begin by restricting the number of non-zero $r_i$, and then restrict the number of $r_i$ which are zero. By Lemma \ref{lem:FamPCMin} we know 

 $$\gamma(G) = \sum_{i=1}^k \Bigl \lceil \frac{3m_i + r_i}{3} \Bigr \rceil =  \sum_{i=1}^k m_i + \sum_{i=1}^k \Bigl\lceil \frac{r_i}{3} \Bigr\rceil.$$
 
 \noindent As $3m+r = \sum\limits_{i=1}^k (3m_i + r_i)$ then $\sum\limits_{i=1}^k m_i = m+\frac{r}{3}- \sum\limits_{i=1}^k \frac{r_i}{3}$ and 
 
 $$\gamma(G) = m+\frac{r}{3}+ \sum_{i=1}^k \Big( \Bigl\lceil \frac{r_i}{3} \Bigr\rceil - \frac{r_i}{3}\Big).$$
 
 \noindent As $\gamma(G) = \gamma(P_{3m+r})$ and $\gamma(P_{3m+r})=\lceil \frac{3m+r}{3} \rceil = m+\lceil \frac{r}{3} \rceil $ then 
 
 $$\sum_{i=1}^k \Big(\lceil \frac{r_i}{3} \rceil - \frac{r_i}{3}\Big) = \lceil \frac{r}{3} \rceil - \frac{r}{3}.$$

\noindent Let $f(r_i) =\big\lceil \frac{r_i}{3} \big\rceil - \frac{r_i}{3}$. As $r_i \in \{0,1,2\}$, then $f(0)=0, f(1)=\frac{2}{3}$, and $f(2)=\frac{1}{3}$. Now consider the number of $r_i \neq 0$ for the cases $r=0,1, \text{ and } 2$

\begin{enumerate}
\item[$\bullet$] If $r=0$ then $\sum f(r_i) = 0$ and no $r_i \neq 0$.
\item[$\bullet$] If $r=1$ then $\sum f(r_i) = \frac{2}{3}$ and at most two $r_i \neq 0$.
\item[$\bullet$] If $r=2$ then $\sum f(r_i) = \frac{1}{3}$ and at most one $r_i \neq 0$.
\end{enumerate}

We now count the $r_i=0$. For a graph $H$, let $g(H) = ord_3(D(H,-3)) - \gamma(H)$. Using Lemma \ref{lem:Cn_ord3}, Lemma \ref{lem:Pn_ord3} and the fact that $\gamma(C_{3m+r})=\gamma(P_{3m+r})= \lceil \frac{3m+r}{3} \rceil$ we can obtain $g(P_{3m+r})$ and $g(C_{3m+r})$:

\vspace{5mm}

\begin{center}
\begin{tabular}{c c c}
$g(P_{3m+r}) = \left\{
        \begin{array}{ll}
           0               & \quad r=0 \\
           0               & \quad r=1 \\
           0 \text{ or } 1 & \quad r=2 \\
        \end{array}
    \right.$

&,&

$g(C_{3m+r}) = \left\{
        \begin{array}{ll}
           1               & \quad r=0 \\
           0 \text{ or } 1 & \quad r=1 \\
           0               & \quad r=2 \\
        \end{array}
    \right.$
\end{tabular}
\end{center}

\vspace{5mm}

\noindent For simplicity we will denote $g(P_{3m+r})$ and $g(C_{3m+r})$ with $g_P(r)$ and $g_C(r)$. Because $G$ is the disjoint union of a path and cycles then $\gamma (G)$ is just the sum of domination numbers of each of the paths and cycles. Similarly $ord_3(D(G,-3))$ is just the sum of the orders of each of its components. From this we get the following equality:

$$g_P(r) = g_P(r_1)+\sum_{i=2}^k g_C(r_i)$$

Now consider the number of $r_i = 0$ for the cases $r=0,1, \text{ and } 2$

\begin{enumerate}
\item[$\bullet$] If $r=0$ then $g_P(r_1)+\sum\limits_{i=2}^k g_C(r_i) = 0$ and no $r_i =0$ for $i \geq 2$.
\item[$\bullet$] If $r=1$ then $g_P(r_1)+\sum\limits_{i=2}^k g_C(r_i) = 0$ and no $r_i =0$ for $i \geq 2$.
\item[$\bullet$] If $r=2$ then $g_P(r_1)+\sum\limits_{i=2}^k g_C(r_i) = 0 \text{ or } 1$ and at most one $r_i =0$ for $i \geq 2$
\end{enumerate}

\noindent Together with the three cases counting the number $r_i \neq 0$, we can easily see there are at most two $r_i$ for $i \geq 2$. Therefore there are at most two cycle components.

\end{proof}

We have narrowed the number of cycle components to two in graphs which are $\mathcal{D}$-equivalent to paths. From Lemma \ref{lem:Cn-1} we know $D(C_n,-1)$. We will now evaluate $D'(C_n,-1)$, $D''(C_n,-1)$, and $D'''(C_n,-1)$ as well as $D(P_n,-1)$, $D'(P_n,-1)$, $D''(P_n,-1)$, and $D'''(P_n,-1)$.

%
%

\begin{lemma} \cite{2014EqCycleAll} For $n \in \mathbb{N}$
\label{lem:C'n-1}

\vspace{2mm}

$D'(C_n,-1) = \left\{
        \begin{array}{rl}
           -n, & \quad n \equiv 0 $ mod $4 \\
            n, & \quad n \equiv 1 $ mod $4 \\
            0, & \quad n \equiv 2 $ mod $4 \\
            0, & \quad n \equiv 3 $ mod $4 \\
        \end{array}
    \right.$
    
\QED
\end{lemma}

\begin{lemma} \cite{2014EqCycleAll} For $n \in \mathbb{N}$
\label{lem:C''n-1}

\vspace{2mm}

$D''(C_n,-1) = \left\{
        \begin{array}{rl}
           \frac{1}{4}\,n \left( n-4 \right),  & \quad n \equiv 0 $ mod $4 \\
          -\frac{1}{2}\,n \left( n-1 \right),  & \quad n \equiv 1 $ mod $4 \\
           \frac{1}{4}\,n \left( n+2 \right),  & \quad n \equiv 2 $ mod $4 \\
            0                                , & \quad n \equiv 3 $ mod $4 \\
        \end{array}
    \right.$
    
\QED
\end{lemma}

The proofs of Lemmas \ref{lem:C'''n-1} -- \ref{lem:P'''n-1} are similar and are left to the reader.

\begin{lemma} For $n \in \mathbb{N}$
\label{lem:C'''n-1}

\vspace{2mm}

$D'''(C_n,-1) = \left\{
        \begin{array}{rl}
           -\frac{1}{16}{n}^{3}+\frac{3}{4}{n}^{2}-2n  			  	,& \quad n \equiv 0 $ mod $4 \\
          \frac{3}{16}{n}^{3}-\frac {9}{8}{n}^{2}+{\frac {15}{16}}n  ,& \quad n \equiv 1 $ mod $4 \\
           -\frac{3}{16}{n}^{3}+\frac{3}{4}n 						 ,& \quad n \equiv 2 $ mod $4 \\
           \frac{1}{16}{n}^{3}+\frac{3}{8}{n}^{2}+{\frac {5}{16}}n  ,& \quad n \equiv 3 $ mod $4 \\
        \end{array}
    \right.$

  \QED
\end{lemma}

\begin{lemma} For $n \in \mathbb{N}$
\label{lem:Pn-1}

\vspace{2mm}

$D(P_n,-1) = \left\{
        \begin{array}{rl}
            1, & \quad n \equiv 0 $ mod $4 \\
           -1, & \quad n \equiv 1 $ mod $4 \\
           -1, & \quad n \equiv 2 $ mod $4 \\
            1, & \quad n \equiv 3 $ mod $4 \\
        \end{array}
    \right.$ 
    
\QED
\end{lemma}

\begin{lemma} For $n \in \mathbb{N}$
\label{lem:P'n-1}

\vspace{2mm}

$D'(P_n,-1) = \left\{
        \begin{array}{rl}
            0, & \quad n \equiv 0 $ mod $4 \\
\frac{n+1}{2}, & \quad n \equiv 1 $ mod $4 \\
            0, & \quad n \equiv 2 $ mod $4 \\
-\frac{n+1}{2},& \quad n \equiv 3 $ mod $4 \\
        \end{array}
    \right.$

\QED
\end{lemma}

\begin{lemma} For $n \in \mathbb{N}$
\label{lem:P''n-1}

\vspace{2mm}

$D''(P_n,-1) = \left\{
        \begin{array}{rl}
           -\frac{1}{8}n(n+4)     ,            & \quad n \equiv 0 $ mod $4 \\
           -\frac{1}{8}(n-1)^2    ,            & \quad n \equiv 1 $ mod $4 \\
            \frac{1}{8}(n+2)^2    ,            & \quad n \equiv 2 $ mod $4 \\
            \frac{1}{8}(n-3)(n+1)   ,          & \quad n \equiv 3 $ mod $4 \\
        \end{array}
    \right.$
    
\QED 
\end{lemma}

\begin{lemma} For $n \in \mathbb{N}$
\label{lem:P'''n-1}

\vspace{2mm}

$D'''(P_n,-1) = \left\{
        \begin{array}{rl}
            \frac{1}{16}n^3-n                        ,   & \quad n \equiv 0 $ mod $4 \\
           -\frac{9}{16}n^2+\frac{3}{8}n+\frac{3}{16} ,  & \quad n \equiv 1 $ mod $4 \\
           -\frac{1}{16}n^3+\frac{1}{4}n              ,  & \quad n \equiv 2 $ mod $4 \\
            \frac{9}{16}n^2+\frac{3}{8}n-\frac{3}{16} ,  & \quad n \equiv 3 $ mod $4 \\
        \end{array}
    \right.$

\QED

\end{lemma}

We now present our main result, the equivalence class of paths. The next theorem will show $[P_{n}] =\{ P_{n}, \widetilde{P_{n}} \}$ for $n \geq 9$. However, first we will discuss the $[P_{n}]$ for $n \leq 8$ as shown in Table \ref{tab:EqP13-8}. For $n \neq 4,7,8$, $[P_{n}] =\{ P_{n}, \widetilde{P_{n}} \}$ ($P_3$ only has one stem, so $P_3$ and $P_3'$ are isomorphic). Recall from the proof of Lemma \ref{lem:EqPUCk}, $D(P_n,-2)=0$ when $n = 4,7,8$. This is evident in Table \ref{tab:EqP13-8} as $P_4$, $P_7$, and $P_8$ are each $\mathcal{D}$-equivalent to graphs with $K_2$ components. Note that $[P_7]$ and $[P_8]$ each have four graphs, however they effectively only have two graphs as the other two graphs are just copies with an irrelevant edge added (the edge between two stems).

\def\c{0.25}
\def\r{1}
\def\m{0}
\def\h{0}
\begin{table}
\begin{center}
\begin{tabular}{|c|c|c|c|}
\hline
$n$ & $D(P_n,x)$ & \multicolumn{2}{ c| }{$[P_n]$} \\
\hline
1  & $x$ & \multicolumn{2}{ c| }{
\scalebox{\c}{
\begin{tikzpicture}
\begin{scope}[every node/.style={circle,thick,draw}]
    \node (1) at (0,0*\m) {};   
\end{scope}

\node[] () [below = 1.5em] at (0,3.5*\h) {};

\end{tikzpicture}}
} \\ \hline
2  & $x^2+2x$ & \multicolumn{2}{ c| }{
\scalebox{\c}{
\begin{tikzpicture}
\begin{scope}[every node/.style={circle,thick,draw}]
    \node (1) at (0,0*\m) {};
    \node (2) at (1,1*\m) {};   
\end{scope}

\node[] () [below = 1.5em] at (0,3.5*\h) {};

\begin{scope}
    \path [-] (1) edge node {} (2);
\end{scope}
\end{tikzpicture}}
} \\ \hline
3  & $x^3+3x^2+x$ & \multicolumn{2}{ c| }{
\scalebox{\c}{
\begin{tikzpicture}
\begin{scope}[every node/.style={circle,thick,draw}]
    \node (1) at (0,0*\m) {};
    \node (2) at (1,1*\m) {};
    \node (3) at (2,2*\m) {};    
\end{scope}

\node[] () [below = 1.5em] at (0,3.5*\h) {};

\begin{scope}
    \path [-] (1) edge node {} (2);
    \path [-] (2) edge node {} (3);
\end{scope}
\end{tikzpicture}}
} \\ \hline
4  & $x^4+4x^3+4x^2$ 
& 
\scalebox{\c}{
\begin{tikzpicture}
\begin{scope}[every node/.style={circle,thick,draw}]
    \node (1) at (0,0*\m) {};
    \node (2) at (1,1*\m) {};
    \node (3) at (2,2*\m) {};   
    \node (4) at (3,3*\m) {};    
\end{scope}

\node[] () [below = 1.5em] at (0,4.5*\h) {};

\begin{scope}
    \path [-] (1) edge node {} (2);
    \path [-] (2) edge node {} (3);
    \path [-] (3) edge node {} (4);
\end{scope}
\end{tikzpicture}}
&
\scalebox{\c}{
\begin{tikzpicture}
\begin{scope}[every node/.style={circle,thick,draw}]
    \node (1) at (0,0*\m) {};
    \node (2) at (1,1*\m) {};
    \node (3) at (2,2*\m) {};   
    \node (4) at (3,3*\m) {};    
\end{scope}

\node[] () [below = 1.5em] at (0,4.5*\h) {};

\begin{scope}
    \path [-] (1) edge node {} (2);
    \path [-] (3) edge node {} (4);
\end{scope}
\end{tikzpicture}}
\\ \hline
5  & $x^5+5x^4+8x^3+3x^2$ 
& 
\scalebox{\c}{
\begin{tikzpicture}
\begin{scope}[every node/.style={circle,thick,draw}]
    \node (1) at (0,0*\m) {};
    \node (2) at (1,1*\m) {};
    \node (3) at (2,2*\m) {};   
    \node (4) at (3,3*\m) {};   
    \node (5) at (4,4*\m) {};
\end{scope}

\node[] () [below = 1.5em] at (0,5.5*\h) {};

\begin{scope}
    \path [-] (1) edge node {} (2);
    \path [-] (2) edge node {} (3);
    \path [-] (3) edge node {} (4);
    \path [-] (4) edge node {} (5);
    
\end{scope}
\end{tikzpicture}}
&
\scalebox{\c}{
\begin{tikzpicture}
\begin{scope}[every node/.style={circle,thick,draw}]
    \node (1) at (0,0*\m) {};
    \node (2) at (1,1*\m) {};
    \node (3) at (2,2*\m) {};   
    \node (4) at (3,3*\m) {};   
    \node (5) at (4,4*\m) {};
\end{scope}

\node[] () [below = 1.5em] at (0,5.5*\h) {};

\begin{scope}
    \path [-] (1) edge node {} (2);
    \path [-] (2) edge node {} (3);
    \path [-] (3) edge node {} (4);
    \path [-] (4) edge node {} (5);
    
    \path [-] (2) edge[bend left=45] node {} (4);
\end{scope}
\end{tikzpicture}}
\\ \hline
6  & $x^6+6x^5+13x^4+10x^3+x^2$
& 
\scalebox{\c}{
\begin{tikzpicture}
\begin{scope}[every node/.style={circle,thick,draw}]
    \node (1) at (0,0*\m) {};
    \node (2) at (1,1*\m) {};
    \node (3) at (2,2*\m) {};   
    \node (4) at (3,3*\m) {};   
    \node (5) at (4,4*\m) {};
    \node (6) at (5,5*\m) {};
\end{scope}

\node[] () [below = 1.5em] at (0,6.5*\h) {};

\begin{scope}
    \path [-] (1) edge node {} (2);
    \path [-] (2) edge node {} (3);
    \path [-] (3) edge node {} (4);
    \path [-] (4) edge node {} (5);
    \path [-] (5) edge node {} (6);
    
\end{scope}
\end{tikzpicture}}
&
\scalebox{\c}{
\begin{tikzpicture}
\begin{scope}[every node/.style={circle,thick,draw}]
    \node (1) at (0,0*\m) {};
    \node (2) at (1,1*\m) {};
    \node (3) at (2,2*\m) {};   
    \node (4) at (3,3*\m) {};   
    \node (5) at (4,4*\m) {};
    \node (6) at (5,5*\m) {};
\end{scope}

\node[] () [below = 1.5em] at (0,6.5*\h) {};

\begin{scope}
    \path [-] (1) edge node {} (2);
    \path [-] (2) edge node {} (3);
    \path [-] (3) edge node {} (4);
    \path [-] (4) edge node {} (5);
    \path [-] (5) edge node {} (6);
    
    \path [-] (2) edge[bend left=45] node {} (5);
\end{scope}
\end{tikzpicture}}
\\ \hline
\multicolumn{1}{ |c  }{\multirow{2}{*}{7} } &
\multicolumn{1}{ |c| }{\multirow{2}{*}{$x^7+7x^6+19x^5+22x^4+8x^3$}}
& 
\scalebox{\c}{
\begin{tikzpicture}
\begin{scope}[every node/.style={circle,thick,draw}]
    \node (1) at (0,0*\m) {};
    \node (2) at (1,1*\m) {};
    \node (3) at (2,2*\m) {};   
    \node (4) at (3,3*\m) {};   
    \node (5) at (4,4*\m) {};
    \node (6) at (5,5*\m) {};
    \node (7) at (6,6*\m) {};
\end{scope}

\node[] () [below = 1.5em] at (0,7.5*\h) {};

\begin{scope}
    \path [-] (1) edge node {} (2);
    \path [-] (2) edge node {} (3);
    \path [-] (3) edge node {} (4);
    \path [-] (4) edge node {} (5);
    \path [-] (5) edge node {} (6);
    \path [-] (6) edge node {} (7);
    
\end{scope}
\end{tikzpicture}}
&
\scalebox{\c}{
\begin{tikzpicture}
\begin{scope}[every node/.style={circle,thick,draw}]
    \node (1) at (0,0*\m) {};
    \node (2) at (1,1*\m) {};
    \node (3) at (2,2*\m) {};   
    \node (4) at (3,3*\m) {};   
    \node (5) at (4,4*\m) {};
    \node (6) at (5,5*\m) {};
    \node (7) at (6,6*\m) {};
\end{scope}

\node[] () [below = 1.5em] at (0,7.5*\h) {};

\begin{scope}
    \path [-] (1) edge node {} (2);
    \path [-] (2) edge node {} (3);
    \path [-] (3) edge node {} (4);
    \path [-] (4) edge node {} (5);
    \path [-] (5) edge node {} (6);
    \path [-] (6) edge node {} (7);
    
    \path [-] (2) edge[bend left=45] node {} (6);
\end{scope}
\end{tikzpicture}}
\\
\cline{3-4}
\multicolumn{1}{ |c  }{}                        &
\multicolumn{1}{ |c| }{} 
& 
\scalebox{\c}{
\begin{tikzpicture}
\begin{scope}[every node/.style={circle,thick,draw}]
    \node (1) at (0*\r,1.5*\r) {};
    \node (2) at (-1.5*\r,0) {};
    \node (3) at (0*\r,-1.5*\r) {};
    \node (4) at (1.5*\r,0) {};

    \node (6) at (-1.426584774*\r, 3-.4635254916*\r) {};
    \node (7) at (1.426584774*\r, 3-.4635254916*\r) {};
    \node (8) at (2.426584774*\r, 4-.4635254916*\r) {};

\end{scope}   

\node[] () [below = 1.5em] at (0,5) {};

\begin{scope}
    \path [-] (1) edge node {} (2);
    \path [-] (2) edge node {} (3);
    \path [-] (3) edge node {} (4);
    \path [-] (4) edge node {} (1);
    
    \path [-] (6) edge node {} (1);
    \path [-] (7) edge node {} (8);
\end{scope}
\end{tikzpicture}}
&
\scalebox{\c}{
\begin{tikzpicture}
\begin{scope}[every node/.style={circle,thick,draw}]
    \node (1) at (0*\r,1.5*\r) {};
    \node (2) at (-1.5*\r,0) {};
    \node (3) at (0*\r,-1.5*\r) {};
    \node (4) at (1.5*\r,0) {};

    \node (6) at (-1.426584774*\r, 3-.4635254916*\r) {};
    \node (7) at (1.426584774*\r, 3-.4635254916*\r) {};
    \node (8) at (2.426584774*\r, 4-.4635254916*\r) {};

\end{scope}   

\node[] () [below = 1.5em] at (0,5) {};

\begin{scope}
    \path [-] (1) edge node {} (2);
    \path [-] (2) edge node {} (3);
    \path [-] (3) edge node {} (4);
    \path [-] (4) edge node {} (1);
    
    \path [-] (6) edge node {} (1);
   \path [-] (7) edge node {} (1);
    \path [-] (7) edge node {} (8);
\end{scope}
\end{tikzpicture}}
\\ \hline
\multicolumn{1}{ |c  }{\multirow{2}{*}{8} } &
\multicolumn{1}{ |c| }{\multirow{2}{*}{$x^8+8x^7+26x^6+40x^5+26x^4+4x^3$}}
& 
\scalebox{\c}{
\begin{tikzpicture}
\begin{scope}[every node/.style={circle,thick,draw}]
    \node (1) at (0,0*\m) {};
    \node (2) at (1,1*\m) {};
    \node (3) at (2,2*\m) {};   
    \node (4) at (3,3*\m) {};   
    \node (5) at (4,4*\m) {};
    \node (6) at (5,5*\m) {};
    \node (7) at (6,6*\m) {};
    \node (8) at (7,7*\m) {}; 
\end{scope}

\node[] () [below = 1.5em] at (0,8.5*\h) {};

\begin{scope}
    \path [-] (1) edge node {} (2);
    \path [-] (2) edge node {} (3);
    \path [-] (3) edge node {} (4);
    \path [-] (4) edge node {} (5);
    \path [-] (5) edge node {} (6);
    \path [-] (6) edge node {} (7);
    \path [-] (7) edge node {} (8);
\end{scope}
\end{tikzpicture}}
&
\scalebox{\c}{
\begin{tikzpicture}
\begin{scope}[every node/.style={circle,thick,draw}]
    \node (1) at (0,0*\m) {};
    \node (2) at (1,1*\m) {};
    \node (3) at (2,2*\m) {};   
    \node (4) at (3,3*\m) {};   
    \node (5) at (4,4*\m) {};
    \node (6) at (5,5*\m) {};
    \node (7) at (6,6*\m) {};
    \node (8) at (7,7*\m) {}; 
\end{scope}

\node[] () [below = 1.5em] at (0,8.5*\h) {};

\begin{scope}
    \path [-] (1) edge node {} (2);
    \path [-] (2) edge node {} (3);
    \path [-] (3) edge node {} (4);
    \path [-] (4) edge node {} (5);
    \path [-] (5) edge node {} (6);
    \path [-] (6) edge node {} (7);
    \path [-] (7) edge node {} (8);
    
    \path [-] (2) edge[bend left=45] node {} (7);
\end{scope}
\end{tikzpicture}}
\\
\cline{3-4}
\multicolumn{1}{ |c  }{}                        &
\multicolumn{1}{ |c| }{} 
& 
\scalebox{\c}{
\begin{tikzpicture}
\begin{scope}[every node/.style={circle,thick,draw}]
    \node (1) at (0*\r,1.5*\r) {};
    \node (2) at (-1.426584774*\r, .4635254916*\r) {};
    \node (3) at (-.8816778783*\r, -1.213525492*\r) {};
    \node (4) at (.8816778783*\r, -1.213525492*\r) {};
    \node (5) at (1.426584774*\r, .4635254916*\r) {};
    
    \node (6) at (-1.426584774*\r, 3-.4635254916*\r) {};
    \node (7) at (1.426584774*\r, 3-.4635254916*\r) {};
    \node (8) at (2.426584774*\r, 4-.4635254916*\r) {};

\end{scope}   

\node[] () [below = 1.5em] at (0,5) {};

\begin{scope}
    \path [-] (1) edge node {} (2);
    \path [-] (2) edge node {} (3);
    \path [-] (3) edge node {} (4);
    \path [-] (4) edge node {} (5);
    \path [-] (5) edge node {} (1);
    
    \path [-] (6) edge node {} (1);
    \path [-] (7) edge node {} (8);
\end{scope}
\end{tikzpicture}}
&
\scalebox{\c}{
\begin{tikzpicture}
\begin{scope}[every node/.style={circle,thick,draw}]
    \node (1) at (0*\r,1.5*\r) {};
    \node (2) at (-1.426584774*\r, .4635254916*\r) {};
    \node (3) at (-.8816778783*\r, -1.213525492*\r) {};
    \node (4) at (.8816778783*\r, -1.213525492*\r) {};
    \node (5) at (1.426584774*\r, .4635254916*\r) {};
    
    \node (6) at (-1.426584774*\r, 3-.4635254916*\r) {};
    \node (7) at (1.426584774*\r, 3-.4635254916*\r) {};
    \node (8) at (2.426584774*\r, 4-.4635254916*\r) {};

\end{scope}   

\begin{scope}
    \path [-] (1) edge node {} (2);
    \path [-] (2) edge node {} (3);
    \path [-] (3) edge node {} (4);
    \path [-] (4) edge node {} (5);
    \path [-] (5) edge node {} (1);
    
    \path [-] (6) edge node {} (1);
    \path [-] (7) edge node {} (1);
    \path [-] (7) edge node {} (8);
\end{scope}
\end{tikzpicture}}
\\ \hline
\end{tabular}
\end{center}
\caption{The domination equivalence classes for paths up to length eight}
\label{tab:EqP13-8}
\end{table}

\begin{theorem}
\label{thm:EqPn}
Let $F_{i}$ ($i\geq 3$) denote the graph that consists of a cycle $C_{i}$ with a pendant edge (that is, one of the vertices $v_{i}$ of the cycle is attached to a new vertex of degree $1$), and $H_{i}$ denote the graph formed from $F_{i}$ and $K_{2}$ by adding in an edge between the stem in $F_{i}$ and a vertex of $K_{2}$.  Then 
\begin{itemize}
\item $[P_{n}] = \{P_{n}\}$ if $n \leq 3$,
\item $[P_{4}] = \{P_{4}, 2P_{2} \}$,
\item $[P_{n}] =\{ P_{n}, \widetilde{P_{n}}, F_{n-3} \cup K_{2}, H_{n-3}\}$ for $n = 7,~8$, and
\item $[P_{n}] =\{ P_{n}, \widetilde{P_{n}} \}$ otherwise.
\end{itemize}
\end{theorem}

\begin{proof}

From previous remarks, it suffices to only consider $n \geq 9$. Let $G$ be a graph which is $\mathcal{D}$-equivalent to $P_n$. By Lemma \ref{lem:EqPCC} $G=H \cup C$ where $H \in \{P_{n_1},\widetilde{P_{n_1}}\}$ with $n_1 \leq n$ and $C$ is the disjoint union of at most two cycles. Therefore either $G=H$, $G=H \cup C_{n_2}$, or  $G=H \cup C_{n_2} \cup C_{n_3}$. It is sufficient to show the latter two cases always yield a contradiction. We will do so by evaluating $D(P_n,-1), \ldots, D'''(P_n,-1)$ and $D(G,-1), \ldots, D'''(G,-1)$ for all cases $n_1,n_2,n_3 \equiv 0,1,2,3 \pmod{4}$ and showing each case contradicts $D(G,x) = D(P_{n},x)$. By Lemma~\ref{lem:Cn-1} there can be no cycles with order congruent to $0 \pmod {4}$, there are $12$ cases to consider for one cycle (without loss $n_1 \equiv 0,1,2,3 \pmod{4}$ and $n_2 \equiv 0,1,2,3 \pmod{4}$) and similarly $24$ cases for two cycles. In each case, we can derive a contradiction. 

We begin with the situation of only one cycle, so that, without loss, $n_3 = 0$ and $G = P_{n_1} \cup C_{n_2}$. 
By Theorem \ref{thm:DomPolyGeo}, taking the first three derivatives of $D(G,x)$ we obtain the following system of equations:

\begin{center}
\begin{tabular}{r c l r}
$D(P_n,-1)$ & $=$ & $ D(P_{n_1},-1)D(C_{n_2},-1)$ & $(PC0)$ \\
		    &     &  \\
$D'(P_n,-1)$ & $=$ & $ D'(P_{n_1},-1)D(C_{n_2},-1)+ D(P_{n_1},-1)D'(C_{n_2},-1)$ & $(PC1)$ \\
		    &     &  \\
$D''(P_n,-1)$ & $=$ & $ D''(P_{n_1},-1)D(C_{n_2},-1)+ 2D'(P_{n_1},-1)D'(C_{n_2},-1)$\\
		    &     & $+ D(P_{n_1},-1)D''(C_{n_2},-1)$ & $(PC2)$ \\
		    &     &  \\
$D'''(P_n,-1)$ & $=$ & $D'''(P_{n_1},-1)D(C_{n_2},-1)+3D''(P_{n_1},-1)D'(C_{n_2},-1)$\\
		    &     & $+ 3D'(P_{n_1},-1)D''(C_{n_2},-1)+ D(P_{n_1},-1)D'''(C_{n_2},-1)$ & $(PC3)$ \\

\end{tabular}
\end{center}

\noindent There are $12$ cases to consider.

\setcounter{CaseCount}{0}
\def\c{5mm}
\vspace{\c}
\addtocounter{CaseCount}{1}
\noindent \textbf{Case \arabic{CaseCount}:}  \emph{$n_1 \equiv 0$, $n_2 \equiv 1 \pmod{4}$}
\vspace{\c}

\noindent As $n \equiv n_1+n_2$ $\pmod{4}$, $n \equiv 1$ $\pmod{4}$, and so equation $(PC1)$ reduces to

$$ \frac{n+1}{2}=n_2 $$

\noindent and equation $(PC2)$ reduces to

$$-\frac{(n-1)^2}{8} = \frac{n_1(n_1+4)}{8}-\frac{n_2(n_2-1)}{2}.$$

\noindent Therefore $n=2n_2-1$. As $n = n_1 + n_2$, $n_1 = n_2-1$. So by substituting this into the reduced equation $(PC2)$ and multiplying both sides by 8 we obtain

$$-(2n_2-2)^2 = (n_2-1)(n_2+3)-4n_2(n_2-1).$$

\noindent Simplifying, we are left with 

$$(n_2-1)^2=0.$$

\noindent Therefore $n_2=1$. As $n_2 \geq 3$, this is a contradiction.

\vspace{\c}
\addtocounter{CaseCount}{1}
\noindent \textbf{Case \arabic{CaseCount}:}  \emph{$n_1 \equiv 0$, $n_2 \equiv 2$ $\pmod{4}$}
\vspace{\c}

\noindent As $n \equiv n_1+n_2$ $\pmod{4}$, $n \equiv 2$ $\pmod{4}$, and so equation $(PC2)$ reduces to

$$\frac{(n+2)^2}{8} = \frac{n_1(n_1+4)}{8}+\frac{n_2(n_2+2)}{4},$$

\noindent and equation $(PC3)$ reduces to

$$-\frac{n^3}{16}+\frac{n}{4} = -\frac{n_1^3}{16}+n_1-\frac{3n_2^3}{16}+\frac{3n_2}{4}.$$

\noindent We will now substitute $n=n_1+n_2$ into the reduced equation $(PC2)$:

\vspace{5mm}

\begin{tabular}{ l c l }
$0$ & $=$ & $\frac{1}{8}(n_1+n_2+2)^2-(\frac{1}{8}n_1(n_1+4)+\frac{1}{4}n_2(n_2+2))$ \\ 
$0$ & $=$ & $(n_1+n_2+2)^2-n_1(n_1+4)-2n_2(n_2+2)$ \\ 
$0$ & $=$ & $n_1^2+n_2^2+4+2n_1n_2+4n_1+4n_2-n_1^2-4n_1-2n_2^2-4n_2$ \\ 
$0$ & $=$ & $-n_2^2+2n_1n_2+4$ \\ 				 
\end{tabular}

\vspace{5mm}

\noindent Therefore $n_1 = \frac{n_2^2-4}{2n_2}$. By substituting this and $n=n_1+n_2$ into the reduced equation $(PC3)$ and multiplying by $n_2$ we obtain

\vspace{5mm}

$$0 = -\frac{1}{4}n_2^4-2n_2^2+12.$$

\noindent We obtain the solutions $n_2 = -2,2,-2\sqrt{3}i \text{ or } 2\sqrt{3}i$. As $n_2$ is order of the cycle, $n_2 \geq 3$ and real, a contradiction for all four solutions.

\vspace{\c}
\addtocounter{CaseCount}{1}
\noindent \textbf{Case \arabic{CaseCount}:}  \emph{$n_1 \equiv 0$, $n_2 \equiv 3$ $\pmod{4}$}
\vspace{\c}

\noindent As $n \equiv n_1+n_2$ $\pmod{4}$, $n \equiv 3$ $\pmod{4}$ and $D(P_n,-1)=1$. However $D(G,-1)=D(P_{n_1},-1)D(C_{n_2},-1)=(1)(-1)=-1$ which is a contraction.

\newpage

\vspace{\c}
\addtocounter{CaseCount}{1}
\noindent \textbf{Case \arabic{CaseCount}:}  \emph{$n_1 \equiv 1$, $n_2 \equiv 1$ $\pmod{4}$}
\vspace{\c}

\noindent, $n \equiv 2$ $\pmod{4}$ and $D(P_n,-1)=-1$. However $D(G,-1)=D(P_{n_1},-1)D(C_{n_2},-1)=(-1)(-1)=1$, which is a contraction.

\vspace{\c}
\addtocounter{CaseCount}{1}
\noindent \textbf{Case \arabic{CaseCount}:}  \emph{$n_1 \equiv 1$, $n_2 \equiv 2$ $\pmod{4}$}
\vspace{\c}

\noindent As $n \equiv n_1+n_2$ $\pmod{4}$, $n \equiv 3$ $\pmod{4}$, and so equation $(PC1)$ reduces to

$$ -\frac{1}{2}(n+1) = -\frac{1}{2}(n_1+1).$$

\noindent This implies $n=n_1$. However $n=n_1+n_2$, so $n_2=0$, which is a contradiction.

\vspace{\c}
\addtocounter{CaseCount}{1}
\noindent \textbf{Case \arabic{CaseCount}:}  \emph{$n_1 \equiv 1$, $n_2 \equiv 3$ $\pmod{4}$}
\vspace{\c}

\noindent As $n \equiv n_1+n_2$ $\pmod{4}$, $n \equiv 0$ $\pmod{4}$, and so equation $(PC1)$ reduces to

$$ 0 = -\frac{1}{2}(n_1+1).$$

\noindent This implies $n_1=-1$, which is a contradiction.

\vspace{\c}
\addtocounter{CaseCount}{1}
\noindent \textbf{Case \arabic{CaseCount}:}  \emph{$n_1 \equiv 2$, $n_2 \equiv 1$ $\pmod{4}$}
\vspace{\c}

\noindent As $n \equiv n_1+n_2$ $\pmod{4}$, $n \equiv 3$ $\pmod{4}$, and so equation $(PC1)$ reduces to

$$ -\frac{n+1}{2}=-n_2, $$

\noindent and equation $(PC2)$ reduces to

$$\frac{1}{8}(n-3)(n+1) = -\frac{1}{8}(n_1+2)^2+\frac{1}{2}n_2(n_2-1).$$

\noindent Therefore $n=2n_2-1$. As $n = n_1 + n_2$, $n_1 = n_2-1$. By substituting this into the reduced equation $(PC2)$ and multiplying both sides by 8 we obtain

$$(2n_2-4)(2n_2) = -(n_2+1)^2+4n_2(n_2-1).$$

\noindent Bringing everything to one side and simplifying we are left with 

$$(n_2-1)^2=0.$$

\noindent Therefore $n_2=1$. However as $n_2 \geq 3$, this is a contradiction.

\vspace{\c}
\addtocounter{CaseCount}{1}
\noindent \textbf{Case \arabic{CaseCount}:}  \emph{$n_1 \equiv 2$, $n_2 \equiv 2$ $\pmod{4}$}
\vspace{\c}

\noindent As $n \equiv n_1+n_2$ $\pmod{4}$, $n \equiv 0$ $\pmod{4}$, and so equation $(PC2)$ reduces to

$$-\frac{1}{8}n(n+4) = -\frac{1}{8}(n_1+2)^2-\frac{1}{4}n_2(n_2+2),$$

\noindent and equation $(PC3)$ reduces to

$$\frac{1}{16}n^3-n = \frac{1}{16}n_1^3-\frac{1}{4}n_1+\frac{3}{16}n_2^3-\frac{3}{4}n_2.$$

\noindent  We will now substitute $n=n_1+n_2$ into the reduced equation $(PC2)$.

\vspace{5mm}

\begin{tabular}{ l c l }
$0$ & $=$ & $-\frac{1}{8}(n_1+n_2)(n_1+n_2+4)-(-\frac{1}{8}(n_1+2)^2-\frac{1}{4}n_2(n_2+2))$ \\ 
$0$ & $=$ & $-(n_1+n_2)(n_1+n_2+4)+(n_1+2)^2+2n_2(n_2+2)$ \\ 
$0$ & $=$ & $-n_1^2-2n_1n_2-n_2^2-4n_1-4n_2+n_1^2+4n_1+4+2n_2^2+4n_2$ \\
$0$ & $=$ & $-2n_1n_2+4+n_2^2$ \\			 
\end{tabular}

\vspace{5mm}

\noindent Therefore $n_1 = \frac{n_2^2+4}{2n_2}$. By substituting this  and $n=n_1+n_2$ into the reduced equation $(PC3)$ and multiplying by $n_2$ we obtain

\vspace{5mm}

$$0 = \frac{1}{4}n_2^4+2n_2^2-12.$$

\noindent We obtain the solutions $n_2 = -2,2,-2\sqrt{3}i \text{ or } 2\sqrt{3}i$. As $n_2$ is order of the cycle, $n_2 \geq 3$ and real. This is a contradiction for all four solutions.

\vspace{\c}
\addtocounter{CaseCount}{1}
\noindent \textbf{Case \arabic{CaseCount}:}  \emph{$n_1 \equiv 2$, $n_2 \equiv 3$ $\pmod{4}$}
\vspace{\c}

\noindent As $n \equiv n_1+n_2$ $\pmod{4}$, $n \equiv 1$ $\pmod{4}$ and $D(P_n,-1)=-1$. However $D(G,-1)=D(P_{n_1},-1)D(C_{n_2},-1)=(-1)(-1)=1$, which is a contraction.

\vspace{\c}
\addtocounter{CaseCount}{1}
\noindent \textbf{Case \arabic{CaseCount}:}  \emph{$n_1 \equiv 3$, $n_2 \equiv 1$ $\pmod{4}$}
\vspace{\c}

\noindent As $n \equiv n_1+n_2$ $\pmod{4}$, $n \equiv 0$ $\pmod{4}$ and $D(P_n,-1)=1$. However $D(G,-1)=D(P_{n_1},-1)D(C_{n_2},-1)=(1)(-1)=-1$, which is a contraction.

\vspace{\c}
\addtocounter{CaseCount}{1}
\noindent \textbf{Case \arabic{CaseCount}:}  \emph{$n_1 \equiv 3$, $n_2 \equiv 2$ $\pmod{4}$}
\vspace{\c}

\noindent As $n \equiv n_1+n_2$ $\pmod{4}$, $n \equiv 1$ $\pmod{4}$ and equation $(PC1)$ reduces to

$$\frac{1}{2}(n+1) = \frac{1}{2}(n_1+1).$$

\noindent This implies $n=n_1$. However $n=n_1+n_2$, so $n_2=0$, which is a contradiction.

\vspace{\c}
\addtocounter{CaseCount}{1}
\noindent \textbf{Case \arabic{CaseCount}:}  \emph{$n_1 \equiv 3$, $n_2 \equiv 3$ $\pmod{4}$}
\vspace{\c}

\noindent As $n \equiv n_1+n_2$ $\pmod{4}$, $n \equiv 2$ $\pmod{4}$ and equation $(PC1)$ reduces to

$$ 0 = \frac{1}{2}(n_1+1).$$

\noindent This implies $n_1=-1$, which is a contradiction.

\vspace{0.1in}

As each of the $12$ cases result in a contradiction, $G$ is not a disjoint union of $H$ and one cycle, where $H \in \{P_{n_1},\widetilde{P_{n_1}}\}$. We will now consider whether $G$ can be a disjoint union of $H$ and two cycles; this yields (without loss) $24$ cases, a number of which can be handled quickly, although some are more involved than the others. Here we have $G = P_{n_1} \cup C_{n_2} \cup C_{n_3}$. By Theorem \ref{thm:DomPolyGeo} taking the first three derivatives of $D(G,x)$ we obtain the following system of equations:

\vspace{5mm}

\begin{center}
\begin{tabular}{r c l r}
$D(P_n,-1)$ & $=$ & $ D(P_{n_1},-1)D(C_{n_2},-1)D(C_{n_3},-1)$ &$(PCC0)$ \\
		    &     &  \\
$D'(P_n,-1)$ & $=$ & $ D'(P_{n_1},-1)D(C_{n_2},-1)D(C_{n_3},-1)+D(P_{n_1},-1)D'(C_{n_2},-1)D(C_{n_3},-1)$ \\
            &     & $+D(P_{n_1},-1)D(C_{n_2},-1)D'(C_{n_3},-1)$&$(PCC1)$ \\
		    &     &  \\
$D''(P_n,-1)$ & $=$ & $ D''(P_{n_1},-1)D(C_{n_2},-1)D(C_{n_3},-1)+D(P_{n_1},-1)D''(C_{n_2},-1)D(C_{n_3},-1)$ \\
            &     & $+D(P_{n_1},-1)D(C_{n_2},-1)D''(C_{n_3},-1)+2D'(P_{n_1},-1)D'(C_{n_2},-1)D(C_{n_3},-1)$&$(PCC2)$  \\
            &     & $+2D'(P_{n_1},-1)D(C_{n_2},-1)D'(C_{n_3},-1)+2D(P_{n_1},-1)D'(C_{n_2},-1)D'(C_{n_3},-1)$\\
            &     &  \\
$D'''(P_n,-1)$ & $=$ & $ D'''(P_{n_1},-1)D(C_{n_2},-1)D(C_{n_3},-1)+D(P_{n_1},-1)D'''(C_{n_2},-1)D(C_{n_3},-1)$ \\
            &     & $+D(P_{n_1},-1)D(C_{n_2},-1)D'''(C_{n_3},-1)+3D''(P_{n_1},-1)D'(C_{n_2},-1)D(C_{n_3},-1)$ \\
            &     & $+3D'(P_{n_1},-1)D''(C_{n_2},-1)D(C_{n_3},-1)+3D''(P_{n_1},-1)D(C_{n_2},-1)D'(C_{n_3},-1)$&$(PCC3)$ \\
            &     & $+3D'(P_{n_1},-1)D(C_{n_2},-1)D''(C_{n_3},-1)+3D(P_{n_1},-1)D''(C_{n_2},-1)D'(C_{n_3},-1)$ \\
            &     & $+3D(P_{n_1},-1)D'(C_{n_2},-1)D''(C_{n_3},-1)+6D'(P_{n_1},-1)D'(C_{n_2},-1)D'(C_{n_3},-1)$ \\

\end{tabular}
\end{center}

\setcounter{CaseCount}{0}
\def\c{5mm}
\vspace{\c}
\addtocounter{CaseCount}{1}
\noindent \textbf{Case \arabic{CaseCount}:}  \emph{$n_1 \equiv 0$, $n_2 \equiv 1$, $n_3 \equiv 1$ $\pmod{4}$}
\vspace{\c}

\noindent As $n \equiv n_1+n_2+n_3$ $\pmod{4}$, $n \equiv 2$ $\pmod{4}$ and by Lemma \ref{lem:Pn-1}, $D(P_n,-1)=-1$. However by Lemma \ref{lem:Cn-1} and Lemma \ref{lem:Pn-1}, $D(P_{n_1},-1)=1$ and $D(C_{n_2},-1)=D(C_{n_3},-1)=-1$ so $D(G,-1)=1$, which is a contradiction.

\vspace{\c}
\addtocounter{CaseCount}{1}
\noindent \textbf{Case \arabic{CaseCount}:}  \emph{$n_1 \equiv 0$, $n_2 \equiv 1$, $n_3 \equiv 2$ $\pmod{4}$}
\vspace{\c}

\noindent As $n \equiv n_1+n_2+n_3$ $\pmod{4}$, $n \equiv 3$ $\pmod{4}$, and so equation $(PCC1)$ reduces to

$$ -\frac{n+1}{2}=-n_2 $$

\noindent and equation $(PCC2)$ reduces to

$$\frac{1}{8}(n-3)(n+1) = -\frac{1}{8}n_1(n_1+4)-\frac{1}{4}n_3(n_3+2)+\frac{1}{2}n_2(n_2-1).$$

\noindent Therefore $n=2n_2-1$. As $n = n_1 + n_2 + n_3$, $n_2 = n_1+n_3+1$ and $n = 2n_1+2n_3+1$. By substituting this into the reduced equation $(PCC2)$ and multiplying both sides by 8 we obtain

$$(2n_1+2n_3-2)(2n_1+2n_3+2) = -n_1(n_1+4)-2n_3(n_3+2)+4(n_1+n_3+1)(n_1+n_3),$$

\noindent which simplifies to

\vspace{5mm}

\begin{tabular}{ r c l }
$4(n_1+n_3)^2-4$ & $=$ & $-n_1(n_1+4)-2n_3(n_3+2)+4(n_1+n_3)^2+4(n_1+n_3)$ \\ 
$-4$             & $=$ & $-n_1^2-4n_1-2n_3^2-4n_3+4n_1+4n_3$ \\
$0$             & $=$ & $-n_1^2-2n_3^2+4.$ \\

\end{tabular}

\vspace{5mm}

\noindent As $n_3 \geq 3$, there are no solutions, which is a contradiction.

\vspace{\c}
\addtocounter{CaseCount}{1}
\noindent \textbf{Case \arabic{CaseCount}:}  \emph{$n_1 \equiv 0$, $n_2 \equiv 2$, $n_3 \equiv 2$ $\pmod{4}$}
\vspace{\c}

\noindent As $n \equiv n_1+n_2+n_3$ $\pmod{4}$, $n \equiv 0$ $\pmod{4}$, and so equation $(PCC2)$ reduces to

$$-\frac{1}{8}n(n+4) = -\frac{1}{8}n_1(n_1+4)-\frac{1}{4}n_3(n_3+2)-\frac{1}{4}n_2(n_2+2).$$

\noindent and equation $(PCC3)$ reduces to

$$\frac{1}{16}n^3-n = \frac{1}{16}n_1^3-n_1+\frac{3}{16}n_3^3-\frac{3}{4}n_3+\frac{3}{16}n_2^3-\frac{3}{4}n_2.$$

\noindent  We will now substitute $n=n_1+n_2+n_3$ into the the reduced equation $(PCC2)$:

$$n_2^2+n_3^2-2n_1n_2-2n_1n_3-2n_2n_3=0.$$

\noindent Therefore if we isolate for $n_1$ we find

$$n_1 = \frac{(n_2-n_3)^2}{2(n_2+n_3)}.$$

\noindent By substituting this and $n=n_1+n_2+n_3$ into the the reduced equation $(PCC3)$, multiplying by $64n_2+64n_3$, and simplifying we obtain

\begin{center}
\[n_2^4-8n_2^3n_3+30n_2^2n_3^2-8n_2n_3^3+n_3^4-16n_2^2-32n_2n_3-16n_3^2=0  \tag{9}\]
\end{center}

\noindent We have plotted the non-negative solutions to equation $(9)$ along with the line $n_3=8-n_2$ in Figure \ref{fig:solP0022}. We will show that any line $n_3=k-n_2$ which intersects the set of non-negative solutions to equation $(9)$ must have $k \leq 8$. Therefore we will be able to bound all solutions to equation $(9)$ with the bounds $n_3 \leq 8-n_2$and  $n_3, n_2 \geq 3$.

\begin{center}
\begin{figure}[h]
\centering
\includegraphics[scale=0.3]{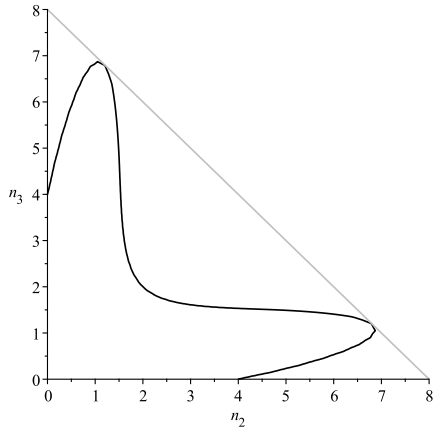}
\caption{Solutions to $n_2^4-8n_2^3n_3+30n_2^2n_3^2-8n_2n_3^3+n_3^4-12n_2^2-32n_2n_3-12n_3^2=0$}
\label{fig:solP0022}
\end{figure}
\end{center}

\noindent  We will show the line $n_3 = k-n_2$ only intersects the set of non-negative solutions to equation $(9)$ if $k \leq 8$. First substitute $n_3 = k-n_2$ into equation $(9)$ to obtain

$$48n_2^4-96kn_2^3+60k^2n_2^2-12k^3n_2+k^4-12k^2=0.$$

\noindent The solutions are

$$n_2 = \frac{1}{2}k \pm \frac{1}{12}\sqrt{18k^2 \pm 6k\sqrt{-3k^2+192}}.$$

\noindent Therefore $n_2$ is real only if $-3k^2+192 \geq 0$ and hence $k \leq 8$. Therefore the only remaining viable solutions are the 6 integer pairs bounded by $n_2,n_3 \geq 3$ and $n_3 \leq 8-n_2$. As none are solutions, this is a contradiction.

\vspace{\c}
\addtocounter{CaseCount}{1}
\noindent \textbf{Case \arabic{CaseCount}:}  \emph{$n_1 \equiv 0$, $n_2 \equiv 1$, $n_3 \equiv 3$ $\pmod{4}$}
\vspace{\c}

\noindent As $n \equiv n_1+n_2+n_3$ $\pmod{4}$, $n \equiv 0$ $\pmod{4}$ and by Lemma \ref{lem:P'n-1}, $D'(P_n,-1)=0$. However by Lemma \ref{lem:C'n-1} and Lemma \ref{lem:P'n-1} $D'(G,-1)=-n_2$, so $n_2=0$, which is a contradiction as $n_2 \geq 3$.

\vspace{\c}
\addtocounter{CaseCount}{1}
\noindent \textbf{Case \arabic{CaseCount}:}  \emph{$n_1 \equiv 0$, $n_2 \equiv 2$, $n_3 \equiv 3$ $\pmod{4}$}
\vspace{\c}

\noindent As $n \equiv n_1+n_2+n_3$ $\pmod{4}$, $n \equiv 1$ $\pmod{4}$ and by Lemma \ref{lem:Pn-1}, $D(P_n,-1)=-1$. However by Lemma \ref{lem:Cn-1} and Lemma \ref{lem:Pn-1}, $D(P_{n_1},-1)=1$ and $D(C_{n_2},-1)=D(C_{n_3},-1)=-1$ so $D(G,-1)=1$, which is a contradiction.

\vspace{\c}
\addtocounter{CaseCount}{1}
\noindent \textbf{Case \arabic{CaseCount}:}  \emph{$n_1 \equiv 0$, $n_2 \equiv 3$, $n_3 \equiv 3$ $\pmod{4}$}
\vspace{\c}

\noindent As $n \equiv n_1+n_2+n_3$ $\pmod{4}$, $n \equiv 2$ $\pmod{4}$ and by Lemma \ref{lem:Pn-1}, $D(P_n,-1)=-1$. However by Lemma \ref{lem:Cn-1} and Lemma \ref{lem:Pn-1}, $D(P_{n_1},-1)=1$ and $D(C_{n_2},-1)=D(C_{n_3},-1)=-1$ so $D(G,-1)=1$, which, is a contradiction.


\vspace{\c}
\addtocounter{CaseCount}{1}
\noindent \textbf{Case \arabic{CaseCount}:}  \emph{$n_1 \equiv 1$, $n_2 \equiv 1$, $n_3 \equiv 1$ $\pmod{4}$}
\vspace{\c}

\noindent As $n \equiv n_1+n_2+n_3$ $\pmod{4}$, $n \equiv 3$ $\pmod{4}$ and by Lemma \ref{lem:Pn-1}, $D(P_n,-1)=1$. However by Lemma \ref{lem:Cn-1} and Lemma \ref{lem:Pn-1}, $D(P_{n_1},-1)=-1$ and $D(C_{n_2},-1)=D(C_{n_3},-1)=-1$ so $D(G,-1)=-1$, which is a contradiction.

\vspace{\c}
\addtocounter{CaseCount}{1}
\noindent \textbf{Case \arabic{CaseCount}:}  \emph{$n_1 \equiv 1$, $n_2 \equiv 1$, $n_3 \equiv 2$ $\pmod{4}$}
\vspace{\c}

\noindent As $n \equiv n_1+n_2+n_3$ $\pmod{4}$, $n \equiv 0$ $\pmod{4}$ and by Lemma \ref{lem:Pn-1}, $D(P_n,-1)=1$. However by Lemma \ref{lem:Cn-1} and Lemma \ref{lem:Pn-1}, $D(P_{n_1},-1)=-1$ and $D(C_{n_2},-1)=D(C_{n_3},-1)=-1$ so $D(G,-1)=-1$, which is a contradiction.

\vspace{\c}
\addtocounter{CaseCount}{1}
\noindent \textbf{Case \arabic{CaseCount}:}  \emph{$n_1 \equiv 1$, $n_2 \equiv 2$, $n_3 \equiv 2$ $\pmod{4}$}
\vspace{\c}

\noindent As $n \equiv n_1+n_2+n_3$ $\pmod{4}$, $n \equiv 1$ $\pmod{4}$, and so equation $(PCC1)$ reduces to

$$ \frac{n+1}{2}= \frac{n_1+1}{2}. $$

\noindent However this implies $n=n_1$, which is a contradiction.

\vspace{\c}
\addtocounter{CaseCount}{1}
\noindent \textbf{Case \arabic{CaseCount}:}  \emph{$n_1 \equiv 1$, $n_2 \equiv 1$, $n_3 \equiv 3$ $\pmod{4}$}
\vspace{\c}

\noindent As $n \equiv n_1+n_2+n_3$ $\pmod{4}$, $n \equiv 1$ $\pmod{4}$, and so equation $(PCC1)$ reduces to

$$ \frac{n+1}{2}= \frac{n_1+1}{2} +n_2$$

\noindent Therefore $n=n_1+2n_2$. Furthermore equation $(PCC2)$ reduces to

$$ -\frac{1}{8}(n-1)^2= -\frac{1}{8}(n_1-1)^2-n_2(n_1+1)-\frac{1}{2}n_2(n_2-1).$$

\noindent By substituting $n=n_1+2n_2$ into the reduced equation $(PCC2)$ and multiplying both sides by 8 we obtain

$$ -(n_1+2n_2-1)^2= -(n_1-1)^2-8n_2(n_1+1)-4n_2(n_2-1).$$

\noindent This simplifies to

\vspace{5mm}

\begin{tabular}{ r c l }
$0$ & $=$ & $-(n_1+2n_2-1)^2-(-(n_1-1)^2-8n_2(n_1+1)-4n_2(n_2-1))$ \\
$0$ & $=$ & $-(n_1+2n_2-1)^2+(n_1-1)^2+8n_2n_1+8n_2+4n_2^2-4n_2$ \\ 
$0$ & $=$ & $-(n_1-1)^2-2(2n_2)(n_1-1)-4n_2^2+(n_1-1)^2+8n_2n_1+4n_2+4n_2^2$ \\ 
$0$ & $=$ & $-2(2n_2)(n_1-1)+8n_2n_1+4n_2$ \\
$0$ & $=$ & $-4n_2n_1+4n_2+8n_2n_1+4n_2$ \\ 
$0$ & $=$ & $4n_2n_1+8n_2$ \\ 		 
\end{tabular}

\vspace{5mm}

\noindent As $n_2 \geq 3$, this equation has no solutions, which is a contradiction.

\vspace{\c}
\addtocounter{CaseCount}{1}
\noindent \textbf{Case \arabic{CaseCount}:}  \emph{$n_1 \equiv 1$, $n_2 \equiv 2$, $n_3 \equiv 3$ $\pmod{4}$}
\vspace{\c}

\noindent As $n \equiv n_1+n_2+n_3$ $\pmod{4}$, $n \equiv 2$ $\pmod{4}$, and so equation $(PCC1)$ reduces to

$$0= \frac{n_1+1}{2}.$$

\noindent Therefore $n_1=-1$, which is a contradiction as $n_1>0$.

\vspace{\c}
\addtocounter{CaseCount}{1}
\noindent \textbf{Case \arabic{CaseCount}:}  \emph{$n_1 \equiv 1$, $n_2 \equiv 3$, $n_3 \equiv 3$ $\pmod{4}$}
\vspace{\c}

\noindent As $n \equiv n_1+n_2+n_3$ $\pmod{4}$, $n \equiv 3$ $\pmod{4}$ and by Lemma \ref{lem:Pn-1}, $D(P_n,-1)=1$. However by Lemma \ref{lem:Cn-1} and Lemma \ref{lem:Pn-1}, $D(P_{n_1},-1)=-1$ and $D(C_{n_2},-1)=D(C_{n_3},-1)=-1$ so $D(G,-1)=-1$, which is a contradiction.

\vspace{\c}
\addtocounter{CaseCount}{1}
\noindent \textbf{Case \arabic{CaseCount}:}  \emph{$n_1 \equiv 2$, $n_2 \equiv 1$, $n_3 \equiv 1$ $\pmod{4}$}
\vspace{\c}

\noindent As $n \equiv n_1+n_2+n_3$ $\pmod{4}$, $n \equiv 0$ $\pmod{4}$ and by Lemma \ref{lem:Pn-1}, $D(P_n,-1)=1$. However by Lemma \ref{lem:Cn-1} and Lemma \ref{lem:Pn-1}, $D(P_{n_1},-1)=-1$ and $D(C_{n_2},-1)=D(C_{n_3},-1)=-1$ so $D(G,-1)=-1$, which is a contradiction.

\vspace{\c}
\addtocounter{CaseCount}{1}
\noindent \textbf{Case \arabic{CaseCount}:}  \emph{$n_1 \equiv 2$, $n_2 \equiv 1$, $n_3 \equiv 2$ $\pmod{4}$}
\vspace{\c}

\noindent As $n \equiv n_1+n_2+n_3$ $\pmod{4}$, $n \equiv 1$ $\pmod{4}$, and so equation $(PCC1)$ reduces to

$$ \frac{n+1}{2}=n_2, $$

\noindent and equation $(PCC2)$ reduces to

$$-\frac{1}{8}(n-1)^2 = \frac{1}{8}(n_1+2)^2+\frac{1}{4}n_3(n_3+2)-\frac{1}{2}n_2(n_2-1).$$

\noindent Therefore $n=2n_2-1$. As $n = n_1 + n_2 + n_3$, $n_2 = n_1+n_3+1$ and $n = 2n_1+2n_3+1$. So by substituting this into the reduced equation $(PCC2)$ and multiplying both sides by 8 we obtain

$$-(2n_1+2n_3)^2 = (n_1+2)^2+2n_3(n_3+2)-4(n_1+n_3+1)(n_1+n_3),$$

\noindent which simplifies to

\vspace{5mm}

\begin{tabular}{ r c l }
$-4(n_1+n_3)^2$ & $=$ & $(n_1+2)^2+2n_3(n_3+2)-4(n_1+n_3)^2-4(n_1+n_3)$ \\ 
$0$             & $=$ & $(n_1+2)^2+2n_3(n_3+2)-4(n_1+n_3)$ \\
$0$             & $=$ & $n_1^2+4n_1+4+2n_3^2+4n_3-4n_1-4n_3$ \\
$0$             & $=$ & $n_1^2+4+2n_3^2.$ \\

\end{tabular}

\vspace{5mm}

\noindent As $n_3 \geq 3$, this equation has no solutions, which is a contradiction.

\vspace{\c}
\addtocounter{CaseCount}{1}
\noindent \textbf{Case \arabic{CaseCount}:}  \emph{$n_1 \equiv 2$, $n_2 \equiv 2$, $n_3 \equiv 2$ $\pmod{4}$}
\vspace{\c}

\noindent As $n \equiv n_1+n_2+n_3$ $\pmod{4}$, $n \equiv 2$ $\pmod{4}$, and so equation $(PCC2)$ reduces to

$$\frac{1}{8}(n+2)^2 = \frac{1}{8}(n_1+2)^2+\frac{1}{4}n_3(n_3+2)+\frac{1}{4}n_2(n_2+2).$$

\noindent and equation $(PCC3)$ reduces to

$$-\frac{1}{16}n^3+\frac{1}{4}n = -\frac{1}{16}n_1^3+\frac{1}{4}n_1-\frac{3}{16}n_3^3+\frac{3}{4}n_3-\frac{3}{16}n_2^3+\frac{3}{4}n_2.$$

\noindent We will now substitute $n=n_1+n_2+n_3$ into the reduced equation $(PCC2)$:

$$-n_2^2-n_3^2+2n_1n_2+2n_1n_3+2n_2n_3=0.$$

\noindent Therefore if we isolate for $n_1$ we find

$$n_1 = \frac{(n_2-n_3)^2}{2(n_2+n_3)}.$$

\noindent By substituting this and $n=n_1+n_2+n_3$ into the reduced equation $(PCC3)$, multiplying by $-64n_2-64n_3$, and simplifying we obtain

$$n_2^4-8n_2^3n_3+30n_2^2n_3^2-8n_2n_3^3+n_3^4+32n_2^2+64n_2n_3+32n_3^2=0$$

\noindent We now substitute $n_3 = k-n_2$ into the equation above to obtain

$$48n_2^4-96kn_2^3+60k^2n_2^2-12k^3n_2+k^4+32k^2=0.$$

\noindent The solutions are

$$n_2 = \frac{1}{2}k \pm \frac{1}{12}\sqrt{18k^2 \pm 6k\sqrt{-3k^2-384}}.$$

\noindent Therefore $n_2$ is real only if $-3k^2-384 \geq 0$. However $-3k^2-384 < 0$ and we have no real solutions for $n_2$, which is a contradiction.

\vspace{\c}
\addtocounter{CaseCount}{1}
\noindent \textbf{Case \arabic{CaseCount}:}  \emph{$n_1 \equiv 2$, $n_2 \equiv 1$, $n_3 \equiv 3$ $\pmod{4}$}
\vspace{\c}

\noindent As $n \equiv n_1+n_2+n_3$ $\pmod{4}$, $n \equiv 2$ $\pmod{4}$ and by Lemma \ref{lem:P'n-1}, $D'(P_n,-1)=0$. However by Lemma \ref{lem:C'n-1} and Lemma \ref{lem:P'n-1} $D'(G,-1)=n_2$, so $n_2=0$, which is a contradiction as $n_2 \geq 3$.

\vspace{\c}
\addtocounter{CaseCount}{1}
\noindent \textbf{Case \arabic{CaseCount}:}  \emph{$n_1 \equiv 2$, $n_2 \equiv 2$, $n_3 \equiv 3$ $\pmod{4}$}
\vspace{\c}

\noindent As $n \equiv n_1+n_2+n_3$ $\pmod{4}$, $n \equiv 3$ $\pmod{4}$ and by Lemma \ref{lem:Pn-1}, $D(P_n,-1)=1$. However by Lemma \ref{lem:Cn-1} and Lemma \ref{lem:Pn-1}, $D(P_{n_1},-1)=-1$ and $D(C_{n_2},-1)=D(C_{n_3},-1)=-1$ so $D(G,-1)=-1$, which is a contradiction.

\vspace{\c}
\addtocounter{CaseCount}{1}
\noindent \textbf{Case \arabic{CaseCount}:}  \emph{$n_1 \equiv 2$, $n_2 \equiv 3$, $n_3 \equiv 3$ $\pmod{4}$}
\vspace{\c}

\noindent As $n \equiv n_1+n_2+n_3$ $\pmod{4}$, $n \equiv 0$ $\pmod{4}$ and by Lemma \ref{lem:Pn-1}, $D(P_n,-1)=1$. However by Lemma \ref{lem:Cn-1} and Lemma \ref{lem:Pn-1}, $D(P_{n_1},-1)=-1$ and $D(C_{n_2},-1)=D(C_{n_3},-1)=-1$ so $D(G,-1)=-1$, which is a contradiction.


\vspace{\c}
\addtocounter{CaseCount}{1}
\noindent \textbf{Case \arabic{CaseCount}:}  \emph{$n_1 \equiv 3$, $n_2 \equiv 1$, $n_3 \equiv 1$ $\pmod{4}$}
\vspace{\c}

\noindent As $n \equiv n_1+n_2+n_3$ $\pmod{4}$, $n \equiv 1$ $\pmod{4}$ and by Lemma \ref{lem:Pn-1}, $D(P_n,-1)=-1$. However by Lemma \ref{lem:Cn-1} and Lemma \ref{lem:Pn-1}, $D(P_{n_1},-1)=1$ and $D(C_{n_2},-1)=D(C_{n_3},-1)=-1$ so $D(G,-1)=1$, which is a contradiction.

\vspace{\c}
\addtocounter{CaseCount}{1}
\noindent \textbf{Case \arabic{CaseCount}:}  \emph{$n_1 \equiv 3$, $n_2 \equiv 1$, $n_3 \equiv 2$ $\pmod{4}$}
\vspace{\c}

\noindent As $n \equiv n_1+n_2+n_3$ $\pmod{4}$, $n \equiv 2$ $\pmod{4}$ and by Lemma \ref{lem:Pn-1}, $D(P_n,-1)=-1$. However by Lemma \ref{lem:Cn-1} and Lemma \ref{lem:Pn-1}, $D(P_{n_1},-1)=1$ and $D(C_{n_2},-1)=D(C_{n_3},-1)=-1$ so $D(G,-1)=1$, which is a contradiction.

\vspace{\c}
\addtocounter{CaseCount}{1}
\noindent \textbf{Case \arabic{CaseCount}:}  \emph{$n_1 \equiv 3$, $n_2 \equiv 2$, $n_3 \equiv 2$ $\pmod{4}$}
\vspace{\c}

\noindent As $n \equiv n_1+n_2+n_3$ $\pmod{4}$, $n \equiv 3$ $\pmod{4}$, and so equation $(PCC1)$ reduces to

$$ -\frac{n+1}{2}= -\frac{n_1+1}{2} $$

\noindent However this implies $n=n_1$, which is a contradiction.

\vspace{\c}
\addtocounter{CaseCount}{1}
\noindent \textbf{Case \arabic{CaseCount}:}  \emph{$n_1 \equiv 3$, $n_2 \equiv 1$, $n_3 \equiv 3$ $\pmod{4}$}
\vspace{\c}

\noindent As $n \equiv n_1+n_2+n_3$ $\pmod{4}$, $n \equiv 3$ $\pmod{4}$, and so equation $(PCC1)$ reduces to

$$ -\frac{n+1}{2}= -\frac{n_1+1}{2} -n_2.$$

\noindent Therefore $n=n_1+2n_2$. Furthermore equation $(PCC2)$ reduces to

$$ \frac{1}{8}(n-3)(n+1)= \frac{1}{8}(n_1-3)(n_1+1)+n_2(n_1+1)+\frac{1}{2}n_2(n_2-1).$$

\noindent By substituting $n=n_1+2n_2$ into the reduced equation $(PCC2)$ and multiplying both sides by 8 we obtain

$$(n_1+2n_2-3)(n_1+2n_2+1)= (n_1-3)(n_1+1)+8n_2(n_1+1)+4n_2(n_2-1),$$

\noindent which simplifies to

$$4n_2n_1+8n_2=0.$$

\noindent As $n_2 \geq 0$, this equation has no non-negative solutions, which is a contradiction.

\vspace{\c}
\addtocounter{CaseCount}{1}
\noindent \textbf{Case \arabic{CaseCount}:}  \emph{$n_1 \equiv 3$, $n_2 \equiv 2$, $n_3 \equiv 3$ $\pmod{4}$}
\vspace{\c}

\noindent As $n \equiv n_1+n_2+n_3$ $\pmod{4}$, $n \equiv 0$ $\pmod{4}$, and so equation $(PCC1)$ reduces to

$$0= -\frac{n_1+1}{2}.$$

\noindent Therefore $n_1=-1$, which is a contradiction as $n_1>0$.

\vspace{\c}
\addtocounter{CaseCount}{1}
\noindent \textbf{Case \arabic{CaseCount}:}  \emph{$n_1 \equiv 3$, $n_2 \equiv 3$, $n_3 \equiv 3$ $\pmod{4}$}
\vspace{\c}

\noindent As $n \equiv n_1+n_2+n_3$ $\pmod{4}$, $n \equiv 1$ $\pmod{4}$ and by Lemma \ref{lem:Pn-1}, $D(P_n,-1)=-1$. However by Lemma \ref{lem:Cn-1} and Lemma \ref{lem:Pn-1}, $D(P_{n_1},-1)=1$ and $D(C_{n_2},-1)=D(C_{n_3},-1)=-1$ so $D(G,-1)=1$, which is a contradiction.

\vspace{5mm}

As each case results in a contradiction, $G$ is not a disjoint union of $H$ and one cycle nor two cycles, where $H \in \{P_{n_1},\widetilde{P_{n_1}}\}$. We conclude $G$ has no cycle components and $G \in \{P_n,\widetilde{P_n}\}$.

\end{proof}

\section{Concluding Remarks}

While we have determined the domination equivalence classes for paths, there are many other families of graphs for which the equivalence classes are unknown. One salient one is cycles. Another extension is to consider trees, where the arguments are likely to be more difficult. Partial results on when trees are dominating unique would be interesting, and may highlight what dominating polynomials can say about graphs in the acyclic setting. 

\pagebreak

\vspace{0.25in}
\noindent {\bf Acknowledgements}
\vspace{0.1in}

\noindent J.I. Brown acknowledges support from the Natural Sciences and Engineering Research Council of Canada (grant application RGPIN 170450-2013). 


\bibliography{mybibfile}

\begin{thebibliography}{10}

\bibitem{2010Char}
S.~Akbari, S.~Alikhani, and Y.~H. Peng.
\newblock {Characterization of graphs using domination polynomials}.
\newblock {\em Eur. J. Comb.}, 31:1714--1724, 2010.

\bibitem{2014EqCycleAll}
S.~Akbari and M.~R. Oboudi.
\newblock {Cycles are determined by their domination polynomials}.
\newblock {\em Ars Comb.}, 116:353--358, 2014.

\bibitem{2013n1}
S.~Alikhani.
\newblock {The Domination Polynomial of a Graph at -1}.
\newblock {\em Graphs Comb.}, 29:1175--1181, 2013.

\bibitem{2009Paths}
S.~Alikhani and Y.~H. Peng.
\newblock {Dominating sets and domination polynomials of paths}.
\newblock {\em Int. J. Math. Math. Sci.}, 10:1--10, 2009.

\bibitem{2011Order10}
S.~Alikhani and Y.~H. Peng.
\newblock {Domination polynomials of cubic graphs of order 10}.
\newblock {\em Turkish J. Math.}, 35:355--366, 2011.

\bibitem{2014Intro}
S.~Alikhani and Y.~H. Peng.
\newblock {Introduction to domination polynomial of a graph}.
\newblock {\em Ars Comb.}, 114:257--266, 2014.

\bibitem{2014EqBipartAll}
B.~M. Anthony and M.~E. Picollelli.
\newblock {Complete r-partite Graphs Determined by their Domination
  Polynomial}.
\newblock {\em Graphs Comb.}, 31:1993--2002, 2015.

\bibitem{dongbook}
F.M. Dong, K.M. Koh, and K.L. Teo.
\newblock {\em {Chromatic Polynomials and Chromaticity of Graphs}}.
\newblock World Scientific, Singapore, 2005.

\bibitem{funddom}
T.W. Haynes, S.T.. Hedetniemi, and P.J. Slater.
\newblock {\em {Fundamentals of domination in graphs}}.
\newblock Marcel Dekker, New York, 2005.

\bibitem{Hooker}
J.N. Hooker, R.S. Garfinkel, and C.K. Chen.
\newblock {Finite dominating sets for network location problems}.
\newblock {\em Oper.\ Res.}, 4:100--118, 1991.

\bibitem{2015EqGen}
S.~Jahari and S.~Alikhani.
\newblock {On D-equivalence classes of some graphs}.
\newblock {\em Bull. Georg. Natl. Acad. Sci.}, 10:2016, 2015.

\bibitem{2012Recurr}
T.~Kotek, J.~Preen, F.~Simon, P.~Tittmann, and M.~Trinks.
\newblock {Recurrence relations and splitting formulas for the domination
  polynomial}.
\newblock {\em Electron. J. Comb.}, 19:1--27, 2012.

\end{thebibliography}

\bibliographystyle{plain}

\end{document}